\documentclass[12pt]{amsart}
\usepackage{microtype}
\usepackage{etex} 
\usepackage[active]{srcltx}
\usepackage{calc,amssymb,amsthm,amsmath,amscd, eucal,ulem}
\usepackage{alltt}
\usepackage{mathtools}
\usepackage{bbding}
\RequirePackage[dvipsnames,usenames]{color}

\input{mabliautoref.sty}
\input{xy}
\xyoption{all}
\input{kmacros3.sty}
\usepackage{tikz}
\usepackage{amsfonts, mathrsfs}
\usepackage[cal=boondox, calscaled=1.05]{mathalfa}
\usepackage{calligra}
\usepackage{amssymb}
\usepackage{stmaryrd} 
\usepackage{dcpic, pictexwd}
\usepackage[left=1in,top=1in,right=1in,bottom=1in]{geometry}
\usepackage{bm}
\usepackage{verbatim}
\usepackage{upgreek}
\usepackage{systeme}
\usepackage{colonequals}

\numberwithin{equation}{theorem}

\newcommand{\kay}{\mathcal{k}}

\renewcommand{\:}{\colon}
\newcommand{\eg}{{\itshape e.g.} }

\newcommand{\p}{\mathfrak{p}}
\newcommand{\q}{\mathfrak{q}}

\DeclareMathOperator{\Adj}{Adj}

\DeclareMathOperator{\Ass}{Ass}
\DeclareMathOperator{\GL}{GL}

\DeclareMathOperator{\Cl}{Cl}

\usepackage{setspace}
\usepackage{hyperref}
\usepackage{enumerate}
\usepackage{enumitem}
\usepackage{graphicx}
\usepackage[all,cmtip]{xy}

\usepackage{verbatim}

\theoremstyle{theorem}

\renewcommand{\sO}{\mathcal{O}}

\renewcommand{\sC}{\mathcal{C}}

\renewcommand{\sH}{\mathcal{H}}

\usepackage{marvosym}

\begin{document}
\title{Singularities of determinantal pure pairs}
\author[J.~Carvajal-Rojas]{Javier Carvajal-Rojas}
\address{KU Leuven\\ Department of Mathematics\\ Celestijnenlaan 200B \\3001 Heverlee\\Belgium \newline\indent
Universidad de Costa Rica\\ Escuela de Matem\'atica\\ San Jos\'e 11501\\ Costa Rica}
\email{\href{mailto:javier.carvajalrojas@epfl.ch}{javier.carvajal-rojas@kuleuven.be}}

\author[A.~Vilpert]{Arnaud Vilpert}
\email{\href{mailto:arnaud.vilpert@epfl.ch}{arnaud.vilpert@proton.me}}

\thanks{The first named author was partially supported by the grants ERC-STG \#804334 and FWO \#G079218N}


\begin{abstract}
Let $X$ be a generic determinantal affine variety over a perfect field of characteristic $p \geq 0$ and $P \subset X$ be a standard prime divisor generator of $\Cl X \cong \bZ$. We prove that the pair $(X,P)$ is purely $F$-regular if $p >0$ and so that  $(X,P)$ is purely log terminal (PLT) if $p=0$ and $(X,P)$ is log $\bQ$-Gorenstein. In general, using recent results of Z. Zhuang and S. Lyu, we show that $(X,P)$ is of PLT-type, \ie there is a $\bQ$-divisor $\Delta$ with coefficients in $[0,1)$ such that $(X,P+\Delta)$ is PLT. 
\end{abstract}
\maketitle

\section{Introduction}

Determinantal varieties are fascinating examples of many situations in algebraic geometry and commutative algebra. For instance, they have KLT-type singularities and projective determinantal varieties are of Fano-type. Likewise, in positive characteristics, they are $F$-regular varieties. From this viewpoint, determinantal varieties are a prominent study case of $F$-singularities and singularities in the minimal model program (MMP).

In studying such singularities, factoriality and more specifically the structure of divisor class groups play a central role. For KLT-type and $F$-regular singularities, we know how to get rid of torsion in their divisor class groups by taking cyclic covers; see \eg \cite[\S5]{CarvajalFiniteTorsors}.\footnote{It is worth noting that one may trivialize divisor class groups by much more sophisticated covers \cite{BraunMoragaIterationOfCoxRings}.} In particular, in studying the structure of divisor class groups, those KLT-type/$F$-regular singularities with free divisor class group are the most interesting ones and determinantal singularities are a premium examples of them. Indeed, the divisor class group of a determinantal singularity is free of rank $1$. Moreover, its generator can be taken to be a certain determinantal prime divisor. This work is concerned with the singularities of such prime divisors generating the divisor class group of determinantal varieties. 

There are many reasons to care about the singularities of these determinantal pure pairs. First, this is the ``right thing to do'' from the MMP standpoint, \eg adjunction. Further, the finite generation of divisor class groups of $F$-regular singularities is an open problem of paramount importance in the theory of $F$-singularities, where the non-torsion case is the unsolved one (see \cite{PolstraMCM,MartinNumberOfGenerators,CarvajalFiniteTorsors}). It is therefore relevant to examine the singularities of prime divisors generating the divisor class groups of $F$-regular singularities. For instance, if they happen to be centers of $F$-purity, we may be able to show some finiteness on them. Finally, this is the setup considered in \cite{CarvajalRojasStaeblerTameFundamentalGroupsPurePairs} to obtain singular versions of Abhyankar's lemma. We aim to answer \cite[Question 2.19]{CarvajalRojasStaeblerTameFundamentalGroupsPurePairs} so that their results regarding tame fundamental groups apply to determinantal pure pairs; see Theorem A below. 

Let $\kay$ be an algebraically closed field of characteristic $p \geq 0$. In what follows, all relative operations (\eg products of varieties and tensor products) and relative notions (\eg varieties, smoothness, projectivity) are defined over $\kay$. 

For positive integers $m,n\geq 1$, let us denote the affine variety of $m\times n$ matrices over $\kay$ by $\bM^{m \times n}$. Of course, $\bM^{m \times n} \cong \bA^{mn}$. For $1\leq t \leq \min\{m,n\}$, let $\bM_t^{m \times n} \subset \bM^{m \times n}$ be the affine variety of $m \times n$ $\kay$-matrices of rank $<t$.  Such affine varieties are collectively known as \emph{(generic) affine determinantal varieties}. To see why, let us recall how their ring of coordinates $R_t^{m \times n} \coloneqq \kay[\bM_t^{m\times n}]$ are presented. Let $\bm{x}=(x_{i,j})_{m \times n}$ be an $m \times n$ matrix of indeterminates and $I_t(\bm{x})$ be the homogeneous ideal of $\kay[\bm{x}]$ generated by all the $t$-minors of $\bm{x}$. For example,
\[
I_2\begin{pmatrix}
 u & v & w\\
 x & y & z
 \end{pmatrix} = (vz-wy,uz-wx,uy-vx) \subset \kay\begin{bmatrix}
 u & v & w\\
 x & y & z
 \end{bmatrix}.
\]
It turns out that:
\[
R_t^{m \times n} = \kay[\bm{x}]/I_t(\bm{x}).
\]

The algebras $R_t^{m\times n}$ are referred to as \emph{(generic) determinantal rings}. Since $I_t$ is a homogeneous prime ideal, $R_t^{m \times n}$ is a standard $\bN$-graded $\kay$-algebra and so it defines a projective variety $\bD_t^{m \times n} \coloneqq \Proj R_t^{m \times n}$; which is referred to as \emph{(generic) projective determinantal varieties}. For instance, $\bD_2^{m \times n}$ is none other than the Segre embedding of $\bP^{m-1} \times \bP^{n-1}$ and the other projective determinantal varieties are obtained as their secant varieties \cite[\S9]{HArrisAlgebraicGeometry}.

The algebras $R_t^{m \times n}$ and therefore the varieties $\bM_t^{m\times n}$, $\bD_t^{m \times n}$ have been extensively studied in the literature; see \cite{BrunsVetterDeterminantalRings} for an account. Perhaps, the most classical reason to study them comes from invariant theory. According to the fundamental theorems of invariant theory; see \cite[II, \S3]{ArbarelloCornalbaGriffithsHArrisGeometryOfAlgebraicCurves}, \cite[Ch. 7]{BrunsVetterDeterminantalRings}, we may think of $R_t^{m \times n}$ as the algebra of invariants of the polynomial algebra 
\[
\kay[(\bm{x})_{m \times (t-1)}] \otimes \kay[(\bm{y})_{(t-1) \times n}] = \kay\begin{bmatrix}
 \bm{x} & \bm{0} \\
 \bm{0} & \bm{y} 
 \end{bmatrix} 
\] 
under the linear action of $\GL_{t-1}(\kay)$ given by:
\[
G \: \begin{pmatrix}
 \bm{x} & \bm{0} \\
 \bm{0} & \bm{y} 
 \end{pmatrix}  \mapsto \begin{pmatrix}
 \bm{x} G^{-1} & \bm{0} \\
 \bm{0} & G\bm{y} 
 \end{pmatrix}, \quad \forall G \in \GL_{t-1}(\kay).
\] 
In more geometric terms, we have the following situation:
 \begin{itemize}
     \item $\GL_{t-1}$ acts on $\bM^{m \times (t-1)} \times \bM^{(t-1) \times n}$ via $G \cdot (M,N) \coloneqq (M G^{-1},GN)$,
     \item $\bM_t^{m \times n}$ is the image of the matrix-multiplication map $\bM^{m \times (t-1)} \times \bM^{(t-1) \times n} \to \bM^{m \times n}$ and the induced map $\bM^{m \times (t-1)} \times \bM^{(t-1) \times n} \to \bM_t^{m \times n}$ realizes the quotient of the above action of $\GL_{t-1}$ on $\bM^{m \times (t-1)} \times \bM^{(t-1) \times n}$.
 \end{itemize}
 
 In case $p = 0$, the above implies that $\bM_t^{m \times n}$ is a reductive quotient and therefore it is a normal integral variety with Cohen--Macaulay singularities by the seminal work of Hochster--Roberts \cite{HochsterRobertsRingsOfInvariants}. In fact, $\bM_t^{m \times n}$ is of KLT-type and so has rational singularities; see \cite{BraunGrebLanglisMoragaReductiveQuotientsOfKLTSingularities,ZhuangDirecSummandsKLTSingularities}. In particular, $\bD_t^{m \times n}$ is a variety of Fano-type; see \cite[Lemma 3.1 (3)]{KollarSingulaitieofMMP}. This illustrates how determinantal varieties fit into modern algebraic geometry: they are prominent examples of Fano-type varieties and their singularities. 
 
 The above can be thought of as an optimal answer for what type of singularities determinantal varieties have in characteristic zero. However, in considering arbitrary characteristics, the optimal answer turns out to be given in positive characteristics. Indeed, if $p>0$ then $\bM_t^{m \times n}$ has \emph{(strongly) $F$-regular singularities}; see \cite{HochsterHunekeTCParameterIdealsAndSplitting}, \cf \cite[\S 8]{MaPolstraLecturesOnFsingularities}.\footnote{This implies that $\bM_t^{m \times n}$ is a normal integral variety with Cohen--Macaulay singularities in characteristic zero by reduction modulo $p$, which is how the proof of the aforementioned Hochster--Roberts' theorem works.}  Moreover, at least if $\bM_t^{m \times n}$ is $\bQ$-Gorenstein (\ie $m=n$), this also implies that $\bM_t^{m \times n}$ has log terminal singularities in characteristic zero; see \cite{TakagiInterpretationOfMultiplierIdeals}.  
 In general, one notices that, if say $m \leq n$, there is a pure morphism $\bM_t^{n \times n} \to \bM_t^{m \times n}$ and then one may use \cite[Theorem 1.1]{ZhuangDirecSummandsKLTSingularities} to conclude that $\bM_t^{m \times n}$ is of KLT-type from $\bM_t^{n \times n}$ being log terminal. In conclusion, the $F$-regularity of determinantal varieties in all positive characteristics is an optimal statement about the singularities of determinantal varieties.

 There is, however, a stronger statement to be worked out as $\bM_t^{m \times n}$ is not factorial when $t\geq 2$. Let $\mathfrak{p}_t^{m \times n} \subset R_t^{m \times n}$ be the ideal generated by the $(t-1)$-minors of an arbitrary set of $t-1$ rows of $\bm{x}$. It turns out that $\p_t^{m \times n}$ is a height-$1$ prime ideal of $R_t^{m \times n}$ whose corresponding prime divisor 
 \[
 P_t^{m \times n}\coloneqq V(\p_t^{m \times n}) \subset \bM_t^{m \times n}
 \] 
 freely generates the divisor class group of $\bM_t^{m \times n}$; see \cite[Corollary 8.4]{BrunsVetterDeterminantalRings}. We are therefore interested in knowing the singularities of the \emph{determinantal pure pair} $(\bM_t^{m \times n},P_t^{m \times n})$. The natural question to ask is whether the pair $(\bM_t^{m \times n},P_t^{m \times n})$ is of PLT-type, \ie whether there is an auxiliary $\bQ$-divisor $\Delta$ on $\bM_t^{m \times n}$ with coefficients in $[0,1)$ such that $(\bM_t^{m \times n},P_t^{m \times n}+\Delta)$ is a \emph{purely log terminal} (PLT) pair. Or, even better, is the pair $(\bM_t^{m \times n},P_t^{m \times n})$ \emph{purely $F$-regular} for all $p>0$?
 
 Let $X$ be a normal integral $\bF_p$-scheme and $P\subset X$ be a prime divisor on $X$. Pure $F$-regularity is a notion of $F$-regularity for pure pairs $(X,P)$. It implies that both $X$ and $P$ are $F$-regular but it is much subtler than that as it is more of a ``simultaneous $F$-regularity''. To define the pure $F$-regularity of a pair $(X,P)$, it is convenient to use the adjoint-like test ideals as defined by \cite{TakagiPLTAdjoint}, \cf \cite{SmolkinSubadditivity}. We denote the $e$-th iterate of the Frobenius morphism on $X$ by $F^e\: X \to X$, which is the spectrum of the ring homomorphism $r \mapsto r^{p^e}$.
 
 \begin{terminology}
 Given an ideal $\fra \subset \sO_X$ and an $\sO_X$-linear map $\phi\: F^e_* \sO_X \to \sO_X$ such that $\phi(F^e_* \fra) \subset \fra$, we say that $\fra$ is a \emph{$\phi$-ideal} or equivalently that $\phi$ is \emph{$\fra$-compatible.}
 \end{terminology}
 
 Recall that the \emph{test ideal} $\uptau(X) \subset \sO_X$ is the smallest non-zero ideal that is a $\phi$-ideal for all $\sO_X$-linear maps $\phi \:F^e_* \sO_X \to \sO_X$ and all $e>0$. One says that $X$ is \emph{$F$-regular} if $\uptau(X)=\sO_X$. Likewise, one defines the \emph{test ideal along $P$} and denote it by $\uptau_{\p}(X,P) \subset \sO_X$ as the smallest ideal that is not contained in $\p \coloneqq \sO_X(-P) \subset \sO_X$ and is a $\phi$-ideal for all $\p$-compatible $\sO_X$-linear maps $\phi \: F^e_* \sO_X \to \sO_X$ and all $e>0$. The pair $(X,P)$ is said to be \emph{purely $F$-regular} if $\uptau_{\p}(X,P) = \sO_X$. It turns out that $\uptau_{\p}(X,P) \subset \uptau(X)$, $\uptau_{\p}(X,P)\sO_P \subset \uptau(P)$, and
 \[
\uptau_{\p}(X,P)_x = \begin{cases}
\uptau_{\p_x}(X_x,P_x) & \text{ if } x \in P\\
 \uptau(X)_x = \uptau(X_x) &\text{ if } x \not\in P,
\end{cases}
\]
where $X_x \coloneqq \Spec \sO_{X,x}$. 

Purely $F$-regular pairs and PLT pairs are deeply interconnected. For instance, purely $F$-regular pairs $(X,P)$ are of PLT-type. Conversely, a pure log pair $(X,P+\Delta)$ is PLT if and only if it is of \emph{purely $F$-regular type.} This means that, after spreading $(X,P+\Delta)$ out to a model $(X_A,P_A+\Delta_A)$ over a finite-type $\bZ$-algebra $A\subset \kay$, the fiber $(X_{\mu},P_{\mu}+\Delta_{\mu})$ is purely $F$-regular over all $\mu$ in a dense open subset of $\Spec A$. For details on this circle of ideas, we recommend the excellent survey \cite{TakagiWatanabeFsingularities}; see \cite[Theorem 4.6]{TakagiWatanabeFsingularities}.

 Our main result is that determinantal pairs $(\bM_t^{m \times n},P_t^{m \times n})$ are purely $F$-regular in positive characteristics and of PLT-type in all characteristics. In particular, we answer \cite[Question 2.19]{CarvajalRojasStaeblerTameFundamentalGroupsPurePairs} affirmatively. To be more precise, we establish the following:
 
\begin{mainthm*}[\autoref{thm.MainTheorem}]
The determinantal pair $(\bM_t^{m \times n},P_t^{m \times n})$ is either:
\begin{itemize}
    \item purely $F$-regular if $p >0$, or
    \item purely log terminal if $p = 0$ and $(\bM_t^{m \times n},P_t^{m \times n})$ is log $\bQ$-Gorenstein (\ie $m=n-1$),
    \item of PLT-type if $p=0$, \ie there is a $\bQ$-divisor $\Delta$ with coefficients in $[0,1)$ such that $(X,P+\Delta)$ is a purely log terminal pair. 
\end{itemize}
\end{mainthm*}

Noteworthy, the pure $F$-regularity of $(\bM_t^{m \times n},P_t^{m \times n})$ was worked out in \cite[Examples 2.13 and 2.15]{CarvajalRojasStaeblerTameFundamentalGroupsPurePairs} in the two easiest cases $(m,n,t) = (2,2,2)$ and $(m,n,t) = (2,3,2)$.

As an application of the Main Theorem, we obtain the following; see \cite[Theorems 4.12 and  5.1]{CarvajalRojasStaeblerTameFundamentalGroupsPurePairs}.

\begin{theoremA*}
The tame fundamental group of $\big(\hat{\bM}_t^{m \times n},\hat{P}_t^{m \times n}\big)$ is trivial, where ``$\widehat{-}$'' denotes completion at the origin. In other words, if $L$ is a finite extension of the function field of $\hat{\bM}_t^{m \times n}$ such that the corresponding normalization is \'etale away from $\hat{P}_t^{m \times n}$ but tamely ramified over the generic point of $\hat{P}_t^{m \times n}$, then $L$ equals the function field of $\hat{\bM}_t^{m \times n}$.
\end{theoremA*}

\subsection*{Outline of the proof of Main Theorem} 
First, we observe that there is a pure morphism $f\:\bM_t^{k \times l}\to \bM_t^{m \times n}$ such that $l=k+1$ and $P_t^{k \times l}$ is the cycle-theoretic pullback of $P_t^{m \times n}$ along $f$. 
From this, one concludes that if $\big(\bM_t^{k \times l}, P_t^{k \times l}\big)$ is purely $F$-regular then so is $(\bM_t^{m \times n}, P_t^{m \times n})$; see \autoref{prop.SplitHomomorphismDescend}. By \cite[Theorem 2.10]{ZhuangDirecSummandsKLTSingularities}, the same holds for being of PLT-type if $p=0$. In other words, in proving our Main Theorem, we may assume that $n=m+1$ and so that $(\bM_t^{m \times n}, P_t^{m \times n})$ is log Gorenstein as $(m-n)P_t^{m \times n}$ is a canonical divisor for $\bM_t^{m \times n}$

In the log Gorenstein case, the second point of our Main Theorem follows from the first one by reduction modulo $p$; see \cite[Theorem 4.6]{TakagiWatanabeFsingularities}, \cite{TakagiPLTAdjoint}. Thus, it suffices to prove the pure $F$-regularity of log Gorenstein determinantal pairs. To do so, we follow the methods in \cite[\S 8, Example 8.12]{MaPolstraLecturesOnFsingularities}. More precisely, we extend Fedder--Watanabe's criterion (see \cite[Theorem 8.1]{MaPolstraLecturesOnFsingularities}, \cf \cite{FedderWatanabe}) to the setup of graded pure log Gorenstein pairs; see \autoref{cor.CheckingForFpurity}. Then, the pure $F$-regularity of $(\bM_t^{m \times n}, P_t^{m \times n})$ with $n=m+1$ is equivalent to its $F$-purity, the $F$-regularity of both $\bM_t^{m \times n}$ and $P_t^{m \times n}$, and the corresponding projective pure pair $(\bD_t^{m \times n}, Q_t^{m \times n})$ having purely $F$-regular singularities. The second and third condition follow straight away from the $F$-regularity of determinantal varieties. The fourth condition is verified by a rather involved induction. Thus, we are left with verifying the $F$-purity of determinantal pairs, which follows from L. Seccia's work \cite{SecciaKnutsonIdealsOfGenericMatrices}.

\subsection*{Acknowledgements} 
 The authors would like to thank Lisa Seccia and Daniel Smolkin for very useful discussions and Lukas Braun and Axel St\"abler for their comments on previous versions of this work. The authors are very grateful to Ziquan Zhuang for kindly including the PLT case in his paper \cite{ZhuangDirecSummandsKLTSingularities}. Finally, the authors thank the anonymous referees for their corrections.

\section{On Pure $F$-regularity}

In this section, we survey the basics on pure $F$-regularity. The notion of pure $F$-regularity is a somewhat straightforward extension of $F$-regularity but the details are not conveniently located in one place in the literature. We provide the relevant details here. However, such details are well-known to experts and the key ideas are already found across the literature. See for instance \cite{SchwedeTestIdealsInNonQGor,BlickleTestIdealsViaAlgebras,BlickleSchwedeTuckerFSigPairs1,HaraWatanabeFRegFPure,TakagiAdjointIdealsAlongClosedSubvars}, to mention a few. We therefore claim no originality in this section. Experts may skip ahead to \autoref{lem.TheLemmaGorensteinCase} and \autoref{cor.CheckingForFpurity}. We use the following convention throughout this section and the rest of this work.

\begin{convention}
All our rings are assumed to be noetherian $F$-finite $\bF_p$-algebras. Further, $0 \in \bN$. In particular, our rings are excellent and admit a canonical module (see \cite{KunzOnNoetherianRingsOfCharP,Gabber.tStruc}), which we denote by $\omega_R$ on a ring $R$. When $R$ is a normal integral domain, we denote by $K_R$ a chosen canonical divisor, \ie $\omega_R \cong R(K_R)$. For $e\in \bN$, we use the shorthand $q \coloneqq p^e$, and further $q'\coloneqq p^{e'}$, $q_0 \coloneqq p^{e_0}$, etc.
\end{convention}

\begin{notation} 
If $R$ is a ring, we write $\sC_{e,R} \coloneqq \Hom_R(F^e_*R,R)$ for all $e \in \bN$. In particular, we make an identification $\sC_{0,R}=R$. If $\phi \in \sC_{e,R}$ and $\phi' \in \sC_{e',R}$, then $\phi \cdot \phi' \in \sC_{e+e',R}$ denotes the map $\phi \circ F^e_* \phi'$. For instance, if $r\in R$ then $\phi \cdot r \in \sC_{e,R}$ is the map $\phi \circ (F^e_* \cdot r)$, which sends $F^e r'$ to $\phi(F^e_* rr')$. This turns $\sC_R \coloneqq \bigoplus_{e \in \bN} \sC_{e,R}$ into a (non-commutative) $\bN$-graded $R$-algebra---the so-called \emph{full Cartier algebra of $R$}. It is worth noting that $\phi \cdot r^q = r \cdot \phi$ for all $r\in R$, all $\phi \in \sC_{e,R}$, and all $e\in \bN$. We write $\sC_{R,+} \coloneqq \bigoplus_{e >0} \sC_{e,R}$. 
\end{notation}

\begin{definition}
Let $R$ be a ring. A \emph{Cartier $R$-algebra} $\sC$ is a graded $R$-subalgebra of $\sC_R$. In that case, we refer to $(R,\sC)$ as a \emph{pair}. We write $\sC_+ \coloneqq \bigoplus_{e >0} \sC_e$. We denote by $I(R,\sC)$ the set of ideals $\fra \subset R$ such that $\phi(F^e_* \fra) \subset \fra$ for all $\phi \in \sC_e$, $e \in \bN$. We refer to an element of $I(R,\sC)$ as a \emph{$\sC$-ideal}.
\end{definition}

\begin{example}
Let $R$ be a ring and $\fra \subset R$ be an ideal. We define the \emph{Cartier $R$-algebra of $\fra$-maps $\sC_R^{[\fra]}$} by $\sC_{e,R}^{[\fra]} \coloneqq \{\phi \in \sC_{e,R} \mid I(R,\phi) \ni \fra\}$.
\end{example}

\begin{example}
Let $R$ be a ring and $\phi \in \sC_{e_0,R}$ be a fixed map for some $e_0>0$. We define the \emph{principal Cartier $R$-algebra $\sC_R^{\phi}$} as the smallest Cartier $R$-algebra containing $\phi$, \ie $\sC_R^{\phi} = \bigcap_{\phi \in \sC \subset \sC_R} \sC$. We write $I(R,\phi) \coloneqq I(R,\sC^{\phi})$ and refer to its elements as \emph{$\phi$-ideals}.
\end{example}

\begin{example}
Let $R$ be a normal integral domain and $\Delta$ be an effective $\bQ$-divisor on $\Spec R$. Following \cite{SchwedeFAdjunction}, to a map $\phi \in \sC_{e,R}$ we associate an effective $\bZ_{(p)}$-divisor $\Delta_{\phi}$ such that $K_R + \Delta_{\phi} \sim_{\bZ_{(p)}} 0$. The Cartier $R$-algebra $\sC_R^{\Delta}$ is defined by
\[
\sC_{e,R}^{\Delta} \coloneqq \{\phi \in \sC_{e,R} \mid \Delta_{\phi} \geq \Delta\} = \Hom_R\big( F^e_* R(\lceil (q-1) \Delta \rceil), R \big).
\]
We write $I(R,\Delta)\coloneqq I(R,\sC^{\Delta})$ and refer to its elements as $\Delta$-ideals.
\end{example}

\begin{remark} \label{rem.CartierAlgebrasMaps}
Let $R$ be a normal integral domain. For a fixed $\phi \in \sC_{e_0,R}$ with $e_0 >0$, we have $\sC_R^{\phi} \subset \sC_R^{\Delta_{\phi}}$ and an equality if $e_0=1$. Further, if $\p \subset R$ is a height-$1$ prime ideal with corresponding prime divisor $P \coloneqq V(\p) \subset \Spec R$, then $\sC_R^{[\p]} = \sC_R^P$.
\end{remark}

\begin{definition} \label{def.Fpurity}
A pair $(R,\sC)$ is said to be \emph{(sharply) $F$-pure} if $\sC_+ R =R$. That is, $\sC_+$ contains a surjective map $\phi \in \sC_e$ for some $e>0$.
\end{definition}

\begin{remark}
Given a pair $(R,\sC)$ and an ideal $\fra \subset R$, then $\sC_+\fra$ denotes the ideal of $R$ generated by elements $\phi(F^e_*a)$ where $\phi \in \sC_e$ for some $e>0$ and $a\in \fra$. That is, $\sC_+\fra = \sum_{e>0} \sC_e\fra$ where $\sC_e \fra \subset R$ denotes the ideal $\sum_{\phi \in \sC_e} \phi(F^e_* \fra)$. In particular, it is not obvious that $\sC_+ R=R$ implies the existence of a surjective map $\phi \in \sC_e$ for some $e>0$, as claimed in \autoref{def.Fpurity}. A proof of this can be found in \eg \cite[Proposition 3.6]{BlickleTestIdealsViaAlgebras} (\cf proof of \autoref{thm.PureFregularityTestelements} below).  
\end{remark}

\begin{proposition} \label{prop.BasicsCompatibleIdeals}
Let $(R,\sC$) be a pair. Then:
\begin{enumerate}
\item $(0), (1) \in I(R,\sC)$.
\item The set $I(R,\sC)$ is closed under sum and intersection.
\item If $\fra \in I(R,\sC)$ then $\fra : \frab \in I(R,\sC)$ for any ideal $\frab \subset R$. In particular, $\Ass_R (R/\fra) \subset I(R,\sC)$ and $\sqrt{\fra} \in I(R,\phi)$.
\item If $(R,\sC)$ is $F$-pure, then $I(R,\sC)$ is a finite lattice of radical ideals.
\end{enumerate}
\end{proposition}
\begin{proof}
    See for instance \cite[Proposition 3.3]{SmithZhang}. The finiteness part in (d) is not obvious; see the paragraph after the proof of \cite[Proposition 3.3]{SmithZhang} and the references therein.
\end{proof}

\begin{definition}
Let $(R,\sC)$ be a pair and $\fra \in I(R,\sC)$. We say that:
\begin{enumerate}
    \item A prime ideal $\p \in I(R,\sC)$ is a \emph{center of $F$-purity} of $(R,\sC)$ if $\sC_+R \not\subset \p$.
    \item The pair $(R,\sC)$ is \emph{non-degenerate along $\fra$} if the minimal primes of $\fra$ are centers of $F$-purity of $(R,\sC)$.
    \item The pair $(R,\sC)$ is \emph{purely $F$-regular along $\fra$} if it is non-degenerate along $\fra$ and every proper ideal in $I(R,\sC)$ is contained in at least one minimal prime of $\fra$. In that case, we refer to the minimal primes of $\fra$ as the \emph{maximal centers of $F$-purity} of $(R,\sC)$.
    \item  The pair $(R,\sC)$ is \emph{(strongly) $F$-regular} if $(R,\sC)$ is purely $F$-regular along $(0)$.
    \item The ring \emph{$R$ is purely $F$-regular along $\fra$ (resp. $F$-regular)} if $(R,\sC^{[\fra]}_R)$ (resp. $(R,\sC_R)$) is purely $F$-regular along $\fra$ (resp. $F$-regular).
\end{enumerate}
\end{definition}

\begin{remark} \label{rem.RegularAmbient}
Note that if $(R,\sC)$ is purely $F$-regular along $\fra$ then $R$ is purely $F$-regular along $\fra$. Further,
if $R$ is purely $F$-regular along $\fra \subset R$ then $R/\fra$ is $F$-regular and the converse holds if $R$ is regular by Kunz's theorem (any map $F^e_* R/\fra \to R/\fra$ lifts to an $\fra$-map $F^e_* R \to R$ if $F^e_*R$ is a projective $R$-module).  
\end{remark}

\begin{remark}
    Let $(R,\sC)$ be a pair. By iterating, we may form a descending chain of $\sC$-ideals ideals $R \supset \sC_+R \supset \sC_+^2 R \supset \sC_+^3 R \supset \cdots$, which turns out to stabilize \cite[Proposition 2.13]{BlickleTestIdealsViaAlgebras}. The stable ideal is often denoted by $\upsigma(R,\sC)$ and referred to as the \emph{non-$F$-pure} ideal of $(R,\sC)$. One readily verifies that $\upsigma(R,\sC)$ and $\sC_+R$ cut out the same closed subset of $\Spec R$ and so they are equal up to radical. Moreover, every sup-ideal of $\sC_+R$ is necessarily a $\sC$-ideal. Thus, the prime ideals in $I(R,\sC)$ are exactly the centers of $F$-purity and the points of $V\big(\upsigma(R,\sC)\big)$. In particular, $F$-purity for $(R,\sC)$ means that all prime ideals in $I(R,\sC)$ are centers of $F$-purity.
\end{remark}

\begin{proposition} \label{prop.MinPrimesareCoprime}
Let $(R,\sC)$ be a pair and $\fra \in I(R,\sC)$. If $(R,\sC)$ is purely $F$-regular along $\fra$ then $(R,\sC)$ is $F$-pure and so $\fra$ is radical. Moreover, the minimal primes of $\fra$ are pairwise coprime. 
\end{proposition}
\begin{proof}
Suppose that $(R,\sC)$ is purely $F$-regular along $\fra$ but not $F$-pure. Then, $\sC_+ R$ is contained in a maximal ideal $\fram \subset R$, which is necessarily in $I(R,\sC)$ as $\phi(F_*^e \fram) \subset \phi(F_*^e R) \subset \sC_+ R \subset \fram$ for all $\phi \in \sC_e$ and all $e>0$. Therefore, $\fram$ is contained in a minimal prime of $\fra$ as $(R,\sC)$ is purely $F$-regular. But then so is $\sC_+R$, which violates that $(R,\sC)$ is non-degenerate along $\fra$. 

To see why the minimal primes of $\fra$ are pairwise coprime, let $\p_1, \ldots, \p_k$ be the minimal primes of $\fra$ and pick
$r_i \in \big(\bigcap_{j \neq i} \p_j\big) \setminus \p_i$ (assuming $k>1$) using prime avoidance. Then, $r \coloneqq \sum_{i=1}^k r_i \notin \bigcup_{\p \in \Ass_R(R/\fra)} \p $. Consider the following ideal $\sC r = Rr + \sC_+r$, which contains $r$ and belongs to $I(R,\sC)$. In particular, $\sC r \not\subset \bigcup_{\p \in \Ass_R(R/\fra)} \p $ and so $\sC r \not\subset \p$ for all $\p \in \Ass_R(R/\fra)$ by prime avoidance. Therefore, $\sC r=R$ by the pure $F$-regularity of $(R,\sC)$ along $\fra$. However, $\sC r = \sC r_1 + \cdots + \sC r_k$. In particular, if $\fram$ is an arbitrary maximal ideal of $R$, we have that $\sC r_i \not\subset \fram$ for some $i = 1,\ldots,k$. On the other hand,  $\sC r_i \subset \p_j$ for all $j \neq i$ and so $\p_j \not\subset \fram$ for all $j \neq i$. In conclusion, $\fram$ contains at most one minimal prime of $\fra$. Since $\fram$ is arbitrary, then the minimal primes of $\fra$ are pairwise coprime.
\end{proof}

The following theorem is arguably the most important result in the theory of $F$-regularity. It has appeared with different levels of generality in the literature, \eg \cite[Theorem 5.10]{HochsterHunekeFRegularityTestElementsBaseChange}, \cite[\S6]{SchwedeFAdjunction}, \cite[Proposition 3.21]{SchwedeTestIdealsInNonQGor}, and \cite[Appendix A]{SmolkinSubadditivity}. The proof of the version we present here can be found in \cite{SmolkinGeneralAdjointIdeals}.

\begin{theorem}\label{them.ExistenceTestIdeals}
Let $(R,\sC)$ be a pair and $\fra \in I(R,\sC)$ be a radical ideal such that $(R,\sC)$ is non-degenerate along $\fra$. There exists $c \notin \bigcup_{\p \in \Ass_R(R/\fra)}\p$ such that: for all $r \notin \bigcup_{\p \in \Ass_R(R/\fra)}\p$ there exists $e>0$ and $\phi \in \sC_e$ such that $\phi(F^e_* r)=c$. Furthermore, the ideal $\uptau_{\fra}(R,\sC) \coloneqq \sC c = Rc + \sC_+ c$ 
has the following properties:
\begin{enumerate}
    \item $\uptau_{\fra}(R,\sC) \not\subset \bigcup_{\p \in \Ass_R(R/\fra)}\p$,
    \item $\uptau_{\fra}(R,\sC) \in I(R,\sC)$, and
    \item if $\frab \in I(R,\sC)$ is such that $\frab \not\subset \bigcup_{\p \in \Ass_R(R/\fra)}\p$ then $\uptau_{\fra}(R,\sC) \subset \frab$.
\end{enumerate}
\end{theorem}

\begin{definition}
We refer to the ideal $\uptau_{\fra}(R,\sC)$ in \autoref{them.ExistenceTestIdeals} as the \emph{test ideal of $(R,\sC)$ along $\fra$}---it is the smallest element of $I(R,\phi)$ that is not contained in any minimal prime of $\fra$. We refer to any such element $c$ in \autoref{them.ExistenceTestIdeals} as a \emph{principal test element of $(R,\sC)$ along $\fra$}. We refer to the elements of $\uptau_{\fra}(R,\sC) \setminus \bigcup_{\p \in \Ass_R(R/\fra)} \p$  as \emph{test elements of $(R,\sC)$ along $\fra$}. When $\fra=(0)$, we write $\uptau(R,\sC)\coloneqq \uptau_{(0)}(R,\sC)$ and refer to it as the \emph{test ideal of $(R,\sC)$}.
\end{definition}

\begin{remark}
If $c$ a principal test element then so is $\phi(F^e_* c)$ for all $\phi \in \sC_e$ and all $e\in \bN$. Thus, test elements are sums of principal test elements and the test ideal is generated by principal test elements.
\end{remark}

The following proposition readily follows from the definitions.

\begin{proposition}
Let $(R,\sC)$ be a pair and $\fra \in I(R,\sC)$ be a radical ideal such that $(R,\sC)$ is non-degenerate along $\fra$. Then,
\[
I(R,\sC) = \{ \frab \in I(R,\sC) \mid \frab \subset \p \text{ for some } \p \in \Ass_R(R/\fra)\} \sqcup \{ \frab \in I(R,\sC) \mid \uptau_{\fra}(R,\sC) \subset \frab\}
\]
so that $(R,\sC)$ is purely $F$-regular along $\fra$ if and only if $\uptau_{\fra}(R,\sC) = R$. 
\end{proposition}

\begin{theorem} \label{thm.PureFregularityTestelements}
Let $(R,\sC)$ be a pair and $\fra \in I(R,\sC)$ be a radical ideal such that $(R,\sC)$ is non-degenerate along $\fra$. Let $b$ be a test element of $(R,\sC)$ along $\fra$. Then, $\uptau_{\fra}(R,\sC) = \sC b = Rb + \sC_+b$.  
Furthermore, $(R,\sC)$ is purely $F$-regular along $\fra$ if and only if $1$ is a principal test element of $(R,\sC)$ along $\fra$.
\end{theorem}
\begin{proof}
Let $r$ be a regular element modulo $\fra$, \ie $r \notin \bigcup_{\p \in \Ass_R(R/\fra)} \p$. 
One readily verifies that $\sC r = Rr + \sC_+r \not\subset \bigcup_{\p \in \Ass_R(R/\fra)} \p$ and $\sC r \in I(R,\sC)$. In other words, $\uptau_{\fra}(R,\sC) \subset \sC r$. Moreover, if $r \in \frab$ for some $\frab \in I(R,\sC)$ then $\sC r \subset \frab$. Then it follows that $\uptau_{\fra}(R,\sC)=\sC r$ if $r\in \uptau_{\fra}(R,\sC)$ (\ie if $r$ is a test element of $(R,\sC)$ along $\fra$).

We show now that $1$ is a principal test element if $(R,\sC)$ is purely $F$-regular along $\fra$. Let $r$ be a regular element modulo $\fra$. Since $\uptau_{\fra}(R,\sC)=R$, then $\sC r=R$. In particular, $\sC_+ \sC r = \sC_+R = R$ and so $\sC_+ r = R$, where we used that $(R,\sC)$ is $F$-pure. Since $\sC_+r = \sum_{e>0}\sum_{\phi \in \sC_e} \phi(F^e_* Rr)$, then there are $e_1,\ldots,e_m >0$ and $\phi_i \in \sC_{e_i}$ for all $i =1,\ldots,m$ such that 
\begin{equation} \label{eqn.1}
\phi_1(F^{e_1}_* Rr) + \cdots + \phi_m(F^{e_m}_* Rr)=R.
\end{equation}
Set $e \coloneqq e_1 \cdots e_m$. 
\begin{claim}
$\phi_1^{e/e_1}(F^{e}_* Rr) + \cdots + \phi_m^{e/e_m}(F^{e}_* Rr) = R
$
\end{claim}
\begin{proof}[Proof of claim]
The equality can be checked locally at every maximal ideal $\fram$. Localizing \autoref{eqn.1} at $\fram$ yields
\[
\phi_1(F^{e_1}_* R_{\fram}r) + \cdots + \phi_m(F^{e_m}_* R_{\fram}r)=R_{\fram}
\]
where $\phi_i \: F^{e_i}_* R_{\fram} \to R_{\fram}$ is the localization of $\phi_i$ at $\fram$, which sends $F^{e_i}_* \frac{r}{u} = F^{e_i}_* \frac{ru^{q_i-1}}{u^{q_i}} = \frac{1}{u} F^{e_i}_* ru^{q_i-1}$ to $\frac{1}{u} \phi_i\big(F^{e_i}_* ru^{q_i-1}\big)$. In particular, for some index $i = i_{\fram} \in \{1,\ldots,m\}$ we have that $\phi_i(F^{e_i}_* R_{\fram}r) = R_{\fram}$ (as the sum of proper ideals in a local ring is proper). Further, $\phi_i^k(F^{e_ik}_* R_{\fram}r) = R_{\fram}$ for all $k\geq 1$ and in particular for $k=e/e_i$. Thus, the required equality holds after localizing at $\fram$.
\end{proof}
The claim means that there are $r_1,\ldots, r_m \in R$ such that:
\[
1 = \phi_1^{e/e_1}(F^{e}_* r_1r) + \cdots + \phi_m^{e/e_m}(F^{e}_* r_mr) = \big(\phi_1^{e/e_1} \cdot r_1+ \cdots + \phi_m^{e/e_m} \cdot r_m \big)(F^e_*r).
\]
We are done after noting that $\phi_1^{e/e_1} \cdot r_1+ \cdots + \phi_m^{e/e_m} \cdot r_m \in \sC_e$.
\end{proof}

\begin{proposition}[{\cite[Proposition 0.10, Lemma 0.14]{SmolkinGeneralAdjointIdeals}}] \label{prop.LocalizationAndRestriction}
Let $(R,\sC)$ be a pair that is non-degenerate along a radical ideal $\fra \in I(R,\phi)$. The following statements hold:
\begin{enumerate}
    \item Let $W \subset R$ be a multiplicatively closed subset such that $W \cap \bigcup_{\p \in \Ass_R(R/\fra)} \p = \emptyset$, then $W^{-1}\uptau_{\fra}(R,\sC) = \uptau_{W^{-1}\fra}(W^{-1}R, W^{-1}\sC)$.
    \item $\uptau_{\fra}(R,\sC)(R/\fra) = \uptau(R/\fra,\sC/\fra)$.
\end{enumerate}
\end{proposition}

The following two propositions are used in the proof of our Main Theorem \autoref{thm.MainTheorem}.

\begin{proposition} \label{prop.SplitHomomorphismDescend}
Let $\theta \: R \to S$ be a ring homomorphism. Let $\p \subset R$ and $\q \subset S$ be prime ideals such that $\p \subset \theta^{-1}(\q)$. Suppose that there exists $T \in \Hom_R(S,R)$ such that $T(1)=1$ and $T(\q) \subset \p$. If $S$ is purely $F$-regular along $\q$ then $R$ is purely $F$-regular along $\p$.
\end{proposition}
\begin{proof}
Suppose that $S$ is purely $F$-regular along $\q$ and
let $r \notin \p$. Since $T(\q) \subset \p$ by hypothesis, $\theta(r) \notin \q$ as $T\big(\theta(r)\big)=r$. Then, there exists $e>0 $ and $\psi \in \sC_{e,S}^{[\q]}$ such that $\psi\big(F^e_*\theta(r)\big)=1$. Define $\phi \in \Hom_R(F^e_*R,R)$ to be the composition:
\[
\phi\:F^e_* R \xrightarrow{F^e_* \theta} F^e_* S \xrightarrow{ \psi } S \xrightarrow{ T} R.
\]
Since $T(1)=1$, one readily sees that $\phi(F^e_*r)=1$. Furthermore, $\phi \in \sC_{e,R}^{[\p]}$ as
\[
\phi(F^e_* \p) = T\Big(\psi\big(F^e_* \theta(\p)\big)\Big) \subset T\big(\psi(F^e_* \q)\big) \subset T(\q) \subset \p,
\]
where we used the hypothesis $\p \subset \theta^{-1}(\q)$ (\ie $\theta(\p) \subset \q$) for the first inclusion.
\end{proof}

\begin{proposition} \label{add_variable}
Let $R$ be a ring and $\p \subset R$ be a prime ideal. If $R$ is purely $F$-regular along $\p$, then $R[t]$ is purely $F$-regular along $\p[t]$.
\end{proposition}

\begin{proof}
Observe that $\p[t] \subset R[t]$ is a prime ideal. Indeed, $R[t]/\p[t]=(R/\p)[t]$ is an integral domain. Let $f = r_0 + r_1t + \cdots + r_nt^n \notin \p[t]$, say $r_i \notin \p$ for some $i \in \{0,\ldots,n\}$. Therefore, there exist  $e>0$ and $\psi \in \sC_{R,e}^{[\p]}$ such that $\psi(F_*^e r_i) =1$. Moreover, we may take $e$ to be arbitrarily large. Indeed, if $\psi(F^e_* r R) = R$ the same holds for all powers of $\psi$. We will take $e$ so that $q > n$. Observe that $\psi$ induces an $R[t]$-linear map $\psi[t] \coloneqq \psi \otimes_R R[t] \: (F_*^eR)[t] \to R[t]$. 

Note that $F^e_* \big( R[t] \big) = \bigoplus_{k=0}^{q-1} (F^e_*R)[t] F^e_* t^k$ as an $(F^e_*R)[t]$-module and so as an $R[t]$-module. Let $\pi_k \: F^e_* \big( R[t] \big) \to (F^e_*R)[t]$ be the corresponding projection. Since $e$ was chosen such that $q>n$, we have $\pi_k(F^e_* f)=F^e_*r_k$ for all $k$. Therefore, the $R[t]$-linear map $\phi \coloneqq \psi[t] \circ \pi_i$ is such that $\phi(F^e_* f) = 1$. It remains to prove that $\p[t] \in I\big(R[t], \phi\big)$. For this, note that $F_*^e \big(\p[t]\big) = \bigoplus_{k=0}^{q-1} (F^e_*\p)[t] F^e_* t^k$ and so $\pi_i\big(F_*^e \big(\p[t]\big)\big) \subset (F^e_*\p)[t]$. The result then follows as $\psi[t]\big((F_*^e \p)[t]\big) \subset \p[t]$.
\end{proof}

\section{Pure $F$-regularity of Pure Pairs} Let us consider the following setup.

\begin{setup} \label{setup.purepairs}
   Let $R$ be a normal integral domain and $\p \subset R$ be a height-$1$ prime ideal with corresponding prime divisor $P \coloneqq V(\p) \subset X \coloneqq \Spec R$. Note that $\p = R(-P)$. In what follows, we will be referring to $F$-singularities of the pure pair $(R,P)=(X,P)$, which refer to the $F$-singularities of the Cartier algebra $\sC_R^P$ of $\p$-maps; see \autoref{rem.CartierAlgebrasMaps}.
\end{setup}

\begin{proposition} \label{test_ideal_localization}
  Work in \autoref{setup.purepairs}. Then, $\p$ is a center of $F$-purity of $(R,P)$. Moreover, $\uptau_{\p}(R,P) \subset \uptau(R)$ and
\[
\uptau_{\p}(R,P)_{\q} = \uptau_{\p}(R,P) R_{\q} = \begin{cases}
\uptau_{\p R_{\q}}(R_{\q},P_{\q}) & \text{ if } \q \in P\\
\uptau(R_{\q}) = \uptau(R)R_{\q} &\text{ if } \q \not\in P.
\end{cases}
\]
\end{proposition}
\begin{proof}
Observe that $\p \in I(R,P)$. Indeed, 
\[
\sC_{e,R}^P = \Hom_R\big(F^e_*R\big((q-1)P\big),R\big) = \Hom_R\big(F^e_*R(-P),R(-P)\big)
\] 
inside $\sC_{e,R}$. Moreover, the set of maps $\phi \in \sC_{e,R}$ whose image is contained in $\p=R(-P)$ is given by $\Hom_R\big(F^e_*R,R(-P)\big) = \Hom_R\big(F^e_*R(qP),R\big)$, which is inside $\sC_{e,R}^P$. In other words, $\p$ not being a center of $F$-purity means that the inclusion $R\big((q-1)P\big)\subset R(qP)$ becomes an equality after taking the pushforward $F^e_*$ and dualizing. However, that would imply that $R\big((q-1)P\big)\subset R(qP)$ is an equality, which is not the case. In conclusion, $\p \in I(R,P)$ is a center of $F$-purity.

The inclusion $\uptau_{\p}(R,P) \subset \uptau(R)$ then means that $\uptau(R)$ is not contained in $\p$. One readily sees that $\uptau(R) \not\subset \p$ as otherwise $\uptau(R_{\p}) \subset \p R_{\p}$ and so $R_{\p}$ is not $F$-regular, but this contradicts that $R_{\p}$ is regular as $R$ is normal and $\height \p =1$.

Finally, let $\q \subset R$ be a  prime ideal. If $\q \in P$, then $\uptau_{\p}(R,P)_{\q}$ equals $\uptau_{\p R_{\q}}(R_{\q},P_{\q})$ by \autoref{prop.LocalizationAndRestriction}. If $\q \not\in P$, then $R((q-1)P)_{\q} = R_{\q}$ and so
\begin{align*}
\big(\sC_{e,R}^P\big)_{\q} = R_{\q} \otimes \Hom_R(F^e_*R((q-1)P),R)
&= \Hom_{R_{\q}}(F^e_*R((q-1)P)_{\q},R_{\q}) \\&= \Hom_{R_{\q}}(F^e_*R_{\q},R_{\q}) = \sC_{e,R_{\q}}.
\end{align*}
Hence, $0\neq \uptau_{\p}(R,P)_{\q} \in I(R_{\q},\sC_{R_{\q}})$ and so $\uptau(R) \subset \uptau_{\p}(R,P)_{\q}$.
\end{proof}

\begin{definition} \label{def.PureFregularityPfPairs}
Work in \autoref{setup.purepairs}. We say that the pure pair $(X,P)$ is \emph{purely $F$-regular} if $R$ is purely $F$-regular along $\p$. We will often use $(R,P)$ and $(X,P)$ interchangeably. 
\end{definition}
We obtain the following by applying \autoref{test_ideal_localization} directly.
\begin{proposition} \label{pro.LocalizatioPureCase}
Work in \autoref{setup.purepairs}. Then, $(R,P)$ is purely $F$-regular if and only if the following two conditions hold:
\begin{enumerate}
    \item $(R_{\q},P_{\q})$ is purely $F$-regular for all $\q \in P$.
    \item $R_{\q}$ is $F$-regular for all $\q \not\in P$.
\end{enumerate}
Furthermore, if $(R,P)$ is purely $F$-regular then both $R$ and $R/\p$ are $F$-regular.
\end{proposition}

\subsubsection{The log Gorenstein case} With notation as in \autoref{setup.purepairs}, the log pair $(R,P)$ is said to be \emph{log Gorenstein} if $K_R+P$ is a Cartier divisor. 

\begin{remark} \label{rem.Principalization}
    Work in \autoref{setup.purepairs}. Consider the $R$-linear map
\[
T^e \: \Hom_R\big(F_*^eR((1-q)P),R \big) \xrightarrow{\subset} \Hom_R(F^e_* R,R) \xrightarrow{\Tr_{F^e}} R
\]
where $\Tr_{F^e}$ is a Frobenius trace on $R$, \ie $\Tr_{F^e} \: \phi \mapsto \phi(F^e_*1)$. Then, $(R,P)$ is purely $F$-regular if and only if for all $r \notin \p$ there is $e>0$ such that the $R$-linear map
\[
T^e \cdot r \: \Hom_R\big(F_*^eR((1-q)P),R \big) \to  R, \quad T \cdot r \: \phi \mapsto T^e(\phi \cdot r) = \phi(F^e_*r)
\]
 is surjective.

Let $\kappa^e_R \: F^e_* R(K_R) \to R(K_R)$ be the \emph{$e$-th Cartier operator}, which is the Frobenius trace on $R(K_R)$. We may use Cartier operators to express $T^e$ as the following $R$-linear map
\[
F_*^e R\big((1-q)(K_R+P)\big) \xrightarrow{\subset} F_*^e R\big((1-q)K_R\big) \xrightarrow{\big(\kappa^e_R \otimes R(-K_R)\big)^{**}} R,
\]
where $(-)^{**}$ denotes reflexification. Twisting it by $R(K_R + P)$ and reflexifying yields
\[
\Psi^e\: F_*^e R(K_R+P) \xrightarrow{\subset} F^e_* R(K_R+qP) \xrightarrow{\big(\kappa^e_R \otimes R(P)\big)^{**}} R(K_R+P),
\]
which let us define a \emph{principal} Cartier structure on $R(K_R+P)$, for $\Psi^e = (\Psi^1)^e$. We set $\Psi \coloneqq \Psi^1$.

Now, the map $T^e \cdot r$ is surjective if and only if it admits a section, for $R$ is a projective module. Moreover, $T^e \cdot r$ admits a section if and only if so does $\Psi^e \cdot r$, as they are obtained from one another by reflexified twists. However, $R(K_R+P)$ is not necessarily a projective module and so the surjectivity of $\Psi^e \cdot r$ does not imply that it admits a section. Of course, $R(K_R+P)$ is projective precisely when it is locally free which means that $(R,P)$ is log Gorenstein. We then conclude the following.
\end{remark}

\begin{proposition} \label{pro.LofGorensteinCase}
    Work in \autoref{setup.LogGoresnteinsSetup} and assume that $(R,P)$ is log Gorenstein. Then, with $\Psi$ as in \autoref{rem.Principalization}, $(R,P)$ is purely $F$-regular if and only if for all $r \notin \p$ there is $e>0$ such that $\Psi^e \cdot r$ is surjective.
\end{proposition}

In the local/graded case, we may use Matlis and local duality to translate \autoref{pro.LofGorensteinCase} as the injectivity of certain Frobenius actions on local cohomology modules. 

\begin{setup} \label{setup.LogGoresnteinsSetup}
Let $(R,\fram,\kay)$ be a normal integral $\kay$-algebra of dimension $d$ that is either: local with maximal ideal $\fram$ and residue field $\kay$ or a standard $\bN$-graded algebra with $R_0=\kay$ and irrelevant ideal $R_+=\fram$. Further, we consider a height-$1$ prime ideal $\p \subset R$ with corresponding prime divisor $P \coloneqq V(\p) \subset \Spec R$, which is further assumed homogeneous in the graded case. In either case, we assume that $-P$ is a canonical divisor (\ie $(R,P)$ is log Gorenstein). Thus, we set $K_R=-P$.
\end{setup}

\begin{remark} \label{rem.PrincipalizationDual}
Working in \autoref{setup.LogGoresnteinsSetup}, the maps $\Psi^e$ in \autoref{rem.Principalization} become $R$-linear maps $\Psi^e\: F^e_* R \to R$. We may interpret this as an equality $\sC_{e,R}^P = \Psi^e \cdot R$. That is, the Cartier algebra $\sC_R^P$ of $P$-maps is principally generated by a map $\Psi \: F_* R \to R$, which is unique up to pre-multiplication by units of $R$. 

Now, using Matlis duality at $\fram$, we may study the surjectivity of maps $\Psi^e\cdot r \: F^e_* R \to R$ by the injectivity of their Matlis duals. Observe that, by local duality, the Matlis dual of $\Psi \: F_* R \to R$ is an $R$-linear map:
\[
G\: H^d_{\fram}\big(R(-P)\big) \to F_*H^d_{\fram}\big(R(-P)\big),
\]
which we treat as a $p$-linear map $G \: H^d_{\fram}\big(R(-P)\big) \to H^d_{\fram}\big(R(-P)\big)$.
\end{remark}

\begin{proposition} \label{prop.PFR}
Work in \autoref{setup.LogGoresnteinsSetup}. Then, with $G$ as in \autoref{rem.PrincipalizationDual}, $(R,P)$ is purely $F$-regular if and only if for all $r \notin \p$ there is $e >0$ such that $r \cdot G^e$ is injective.
\end{proposition}

With the above in place, we are ready to provide our pure log Gorenstein version of \cite{FedderWatanabe}. We follow the presentation in \cite[Theorem 8.1]{MaPolstraLecturesOnFsingularities}. In fact, our proof is \emph{mutatis mutandis} the one in \emph{loc. cit.} and we work it out here for the sake of completeness. 

\begin{proposition} \label{lem.TheLemmaGorensteinCase}
Work in the graded case of \autoref{setup.LogGoresnteinsSetup}. Then, $(R,P)$ is purely $F$-regular if and only if the following conditions hold:
\begin{enumerate}
    \item[(a)] $(R_{\q}, P_{\q})$ is purely $F$-regular for all homogeneous prime ideals $\p \subset \q \subsetneq \fram$,
    \item[(a')] $R_{\q}$ is $F$-regular for all homogeneous prime ideals $\p \not\subset \q \subsetneq \fram$,
    \item[(b)] $\big[H_{\fram}^d\big(R(-P)\big)\big]_n=0$ for all $n\geq 0$,
    \item[(c)] $(R,P)$ is $F$-pure. 
\end{enumerate}
\end{proposition}
\begin{proof}
If $(R,\p)$ is purely $F$-regular then (a), (a') and (c) hold by \autoref{pro.LocalizatioPureCase} and \autoref{prop.MinPrimesareCoprime}. For (b), note that both $R$ and $R/\p$ are $F$-regular by \autoref{pro.LocalizatioPureCase}. In particular, $R$ is Cohen--Macaulay and so there is a canonical exact sequence
\[
0 \to H_{\fram}^{d-1}(R/\p) \to E\coloneqq H_{\fram}^d\big(R(-P)\big) \to H_{\fram}^d(R) \to 0
\]
of graded $R$-modules. In particular, (b) is equivalent to the vanishings $[H^{d-1}_{\fram}(R/\p)]_n=0$ and $[H^d_{\fram}(R)]_n=0$ for all $n \geq 0$. However, these vanishings follow from $R$ and $R/\p$ are being both $F$-regular and so $F$-rational.\footnote{One can prove (b) directly by using the same argument that shows the $F$-rationality of $F$-regular rings.} 

Conversely, suppose that (a), (a'), (b), and (c) hold. To apply \autoref{prop.PFR}, let us define the following graded submodules of $E \coloneqq H^d_{\fram}\big(R(-P)\big)$:
\[
M_e(c)  \coloneqq \ker (c \cdot G^e) = \{\eta \in E \mid c G^e \eta =0\}
\]
for all $c \notin \p$. Observe that condition (c) means that $M_e(1)=0$ for some $e>0$, and so $M_1(1)=0$ as $M_1(1) \subset M_e(1)$. In other words, condition (c) means that $G$ is injective. This further implies that we have the following descending chain of $R$-submodules of $E$
\[
E \supset M_1(c) \supset M_2(c) \supset M_3(c) \supset \cdots
\] 
Indeed, if $cG^{e+1}\eta =0$ then $c^pG^{e+1}\eta =0$. However, $c^pG^{e+1}\eta = c^pGG^{e}\eta=G(cG^e\eta)$. Thus, $cG^e\eta=0$ as $G$ is injective. Summing up, $M_{e+1}(c) \subset M_e(c)$.

Since $E$ is artinian, the above descending chain stabilizes. Let $M(c)=\bigcap_{e>0}M_e(c)$ be the stable submodule, \ie $M(c) = M_e(c)$ for all $e \gg 0$. It suffices to prove that $M(c)=0$. To this end, we observe the  following.

\begin{claim} \label{claim.FiniteLength}
$M(c)$ is a graded $R$-module of finite length.
\end{claim}
\begin{proof}[Proof of claim]
    Conditions (a) and (a') imply that for all homogeneous prime ideal $\q\subsetneq \fram$ there is $e_{\q}$ such that 
    \[
    R_{\q} \otimes M_{e_{\q}}(c)= \ker(c \cdot G_{\q}^{e_{\q}})=\big\{\eta \in E_{\q} = H^d_{\fram}\big(R_{\q}(-P_{\q})\big) \bigm| cG_{\q}^{e_{\q}} \eta = 0\big\}=0,
    \]
    where $G_{\q}$ becomes the standard Frobenius action on $H^d_{\fram}(R_{\q})$ for the Gorenstein ring $R_{\q}$ whenever $\p \not\subset \q$. In particular, $R_{\q}\otimes M(c)=0$ for all homogeneous prime ideal $\q\subsetneq \fram$. In other words, $\Supp M(c) \subset \{\fram\}$. If $M(c)$ were finitely generated we would be done. To bybass this, one proceeds as follows. Note that $M(c)$ is the Matlis dual of the cokernel of $\Psi^e \cdot c$ for all $e \gg 0$. Hence, $\Supp \coker (\Psi^e \cdot c) \subset \{\fram\}$ for all $e \gg 0$. Let $e_0$ be one of those such $e \gg 0$. Since $\coker (\Psi^{e_0} \cdot c)$ is finitely generated, we conclude that $\fram^k \coker (\Psi^{e_0} \cdot c)=0$. Therefore, $\fram^k M_{e_0}(c)=\fram^k M(c)=0$ and the claim follows.
\end{proof}

\begin{claim} \label{claim.GStability}
$M(c) \subset E$ is a Frobenius submodule, \ie $G(M(c)) \subset M(c)$.
\end{claim}
\begin{proof}[Proof of claim]
Let $\eta \in M(c)$, \ie  $c G^e \eta =0$  for all  $e>0$. Note that $c G^e G\eta$ = $cG^{e+1} \eta = 0$ for all $e>0$.
\end{proof}

Since $G$ is injective, \autoref{claim.FiniteLength} and \autoref{claim.GStability} imply that $M(c)$ must concentrate in degree $0$ as $G \eta$ is homogeneous of degree $pn$ if $\eta$ is homogeneous of degree $n$. Thus, $[M(c)]_n=0$ unless $n=0$. However, $[M(c)]_0$ is zero by condition (b). 
\end{proof}

\begin{corollary} \label{cor.CheckingForFpurity}
With notation as in \autoref{lem.TheLemmaGorensteinCase}, $(R,P)$ is purely $F$-regular if and only if the following conditions hold
\begin{itemize}
    \item[(a)] $(R_{\q}, P_{\q})$ is purely $F$-regular for all homogeneous prime ideals $\p \subset \q \neq \fram$,
    \item[(b)] both $R$ and $R/\p$ are $F$-regular, and
    \item[(c)] $(R,P)$ is $F$-pure. 
\end{itemize}
\end{corollary}
\begin{remark} \label{rem.SettingVariablesEqualTo1}
    In \autoref{cor.CheckingForFpurity}, we may replace (a) by saying that each localization $(R_{x_i},P_{x_i})$ is purely $F$-regular for all $x_i \notin \p$ where $\fram=(x_1,\ldots,x_n)$. For $x \in R_1$, let $\big(R_{(x)},P_{(x)}\big)$ be the degree $0$ component of $(R_{x},P_{x})$. Then, $R_{x} = \bigoplus_{i \in \bZ} R_{(x)} x^i = R_{(x)}[x,x^{-1}]$ and likewise $\p R_x = \p_{(x)} \big[x,x^{-1}\big]$ when $x \notin \p$. Thus, by \autoref{prop.SplitHomomorphismDescend}, \autoref{add_variable}, and \autoref{test_ideal_localization}, $(R_x, P_x)$ is purely $F$-regular if and only if so is $\big(R_{(x)}, P_{(x)}\big)$. In other words, we may replace (a) by saying that $\big(R_{(x_i)}, P_{(x_i)}\big)$ is purely $F$-regular for all $i=1,\ldots,n$. In more geometric terms, this means that we may replace (a) by saying that the projective pure pair obtained by taking $\Proj$ of $(R,\p)$ has purely $F$-regular singularities. 
    
    In practice, let $R=\kay[x_1,\ldots,x_n]/(f_1,\dots,f_m)$ and $\p=(g_1,\ldots, g_k)$ with $f_i,g_j$ homogeneous polynomials. Then, verifying condition (a) means proving the pure $F$-regularity of the pair
    \[
    \big(\kay[x_1,\ldots,x_n]/(f_1,\dots,f_m,x_i-1), (g_1,\ldots, g_k,x_i-1) \big)
    \]
    for all $i=1,\ldots,n$. Of course, this simply means setting $x_i=1$ in all of our equations.
\end{remark}

\section{On Determinantal Pure Pairs} \label{subsec.DeterminantalPairs} Let $m,n \geq t$ be positive integers. Let $\bm{x}=(x_{i,j})$ be an $m\times n$ matrix of indeterminates and $\kay$ be a field of characteristic $p \geq 0$. Consider the polynomial algebra 
\[
S \coloneqq \kay[\bm{x}]\coloneqq \kay[x_{i,j} \mid 1\leq i \leq m, 1 \leq j \leq n].
\]
A $t$-minor of $\bm{x}$ is the element of $S$ given by the determinant of a $t \times t$ submatrix of $\bm{x}$. We let $I_t(\bm{x}) \subset S$ be the ideal generated by all the $t$-minors of $\bm{x}$, which turns out to be a homogeneous prime ideal of $S$ of height $(m-t+1)(n-t+1)$. One defines the \emph{(generic) determinantal $\kay$-algebra} as the quotient
\[
R_t = R^{m\times n}_t \coloneqq S/I_t(\bm{x}),
\]
which is a standard graded $\kay$-algebra and moreover:

\begin{theorem}[{\cite[Proposition 1.1. and Corollary 5.17.]{BrunsVetterDeterminantalRings}}]
The determinantal ring $R^{m\times n}_t$ is a normal integral domain of dimension $(t-1)(m+n-t+1)$. Moreover, $R^{m\times n}_t$ has Cohen--Macaulay singularities.
\end{theorem}

One defines the determinantal varieties $\mathbb{M}_t=\mathbb{M}_t^{m \times n}$ as $\Spec R_t$. However, for notation ease, we will formulate our statements using the algebras $R_t$ rather than the varieties $\mathbb{M}_t$.

Let us describe next the divisor class group of $R^{m\times n}_t$ for $t \geq 2$. For further details see \cite[\S8]{BrunsVetterDeterminantalRings}. Let $\bm{x}'$ be a $(t-1) \times n$ submatrix of $\bm{x}$ obtained by taking any subset of $t-1$ rows. Which choice of rows we make to define $\bm{x}'$ is irrelevant but, for the sake of concreteness, our preferred choice are the first $t-1$ rows. The ideal $I_{t-1}(\bm{x}')\subset S$ is a prime ideal of height $n-t+2$. Moreover, $I_{t-1}(\bm{x}')+I_t(\bm{x}) \subset S$ is a prime ideal of height $1+\height I_t(\bm{x})$. That is, 
\[
\p_t = \p_t^{m \times n} \coloneqq I_{t-1}(\bm{x}')R_t^{m \times n} \subset R_t^{m \times n}\]
is a prime ideal of height $1$. Thus, 
\[
P_t= P_t^{m\times n} \coloneqq V(\p_t^{m \times n}) \subset \Spec R^{m\times n}_t\] 
defines a prime divisor. The divisor class group $\Cl R_t$ is freely generated by the divisor class of $P_t$. Moreover, a canonical divisor of $R_t$ is given by $K_{R_t}=(m-n)P_t$. In particular, $(R_t,P_t)$ is log Gorenstein  if and only if $m=n-1$.

\begin{definition}
We refer to $(R_t,P_t)=(R_t^{m\times n},P_t^{m \times n})$ as a \emph{determinantal pure pair}.
\end{definition}

\begin{remark}
    We could have defined $P_t$ using columns instead of rows giving us prime divisors $Q_t$. Then, $Q_t \sim -P_t$ and $K_{R_t} = mP_t + nQ_t$. Of course, our results apply to $(R_t,Q_t)$ as well.
\end{remark}

\begin{remark}
    Let $A$ be a general commutative ring with unity and $M$ be an $m \times n$ matrix over $A$. It makes sense to define the ideal $I_t(M) \subset A$ of $t$-minors of $M$ for all $1\leq t \leq \min\{m,n\}$. It is well-known that $I_t(M)$ is an invariant of the corresponding linear transformation $A^{\oplus n} \to A^{\oplus m}$ rather than the matrix $M$ itself. In other words, $I_t(M)=I_t(UMV)$ as ideals of $A$ for all invertible matrices $U\in \GL_m(A)$ and $V\in \GL_n(A)$. See \cite[III, \S8, Proposition 10]{BourbakiAlgebre} and \cite[p. 3]{BrunsVetterDeterminantalRings}. In particular, performing (invertible) elementary row and column operations on $M$ does not affect the ideal $I_t(M)$. We will use freely this fact for the rest of this paper.
\end{remark}

\subsection{$F$-splittings on determinantal varieties}
Next, we explain the existence of a polynomial $\Delta\in S$ such that $\phi \coloneqq \Phi \cdot \Delta^{p-1} \: F_*S \to S$ induces a Frobenius splitting on all determinantal varieties, where $\Phi$ is a Frobenius trace on $S$. The results in this section are due to L. Seccia; see \cite{SecciaKnutsonIdealsOfGenericMatrices}. In fact, she obtains more general statements.

Set $\mu \coloneqq \min\{m,n\}$. The matrix $\bm{x}$ has $\mu-1$ \emph{lower diagonals}, $|n-m|+1$ \emph{main diagonals}, and $\mu-1$ \emph{upper diagonals}. To the lower diagonals, we associate minors $\alpha_1,\ldots,\alpha_{\mu-1}$, where $\alpha_i$ is an $i$-minor. Likewise, to the upper diagonals, we associate the minors $\beta_1,\ldots, \beta_{\mu-1}$, where $\beta_i$ is an $i$-minor. That is,
\[
\alpha_1 \coloneqq x_{m,1}, \alpha_2 \coloneqq \begin{vmatrix}
x_{m-1,1} & x_{m-1,2} \\
x_{m,1} & x_{m,2}
\end{vmatrix}, \alpha_3 \coloneqq \begin{vmatrix}
x_{m-2,1} & x_{m-2,2} & x_{m-2,3}\\
x_{m-1,1} & x_{m-1,2} & x_{m-1,3} \\
x_{m,1} & x_{m,2} & x_{m,3}
\end{vmatrix}, \cdots
\]
and
\[
\beta_1 \coloneqq x_{1,n}, \beta_2 \coloneqq \begin{vmatrix}
x_{1,n-1} & x_{1,n} \\
x_{2,n-1} & x_{2,n}
\end{vmatrix}, \beta_3 \coloneqq \begin{vmatrix}
x_{1,n-2} & x_{1,n-1} & x_{1,n}\\
x_{2,n-2} & x_{2,n-1} & x_{2,n} \\
x_{3,n-2} & x_{3,n-1} & x_{3,n}
\end{vmatrix}, \cdots
\]
Finally, to the $|n-m|+1$ main diagonals we associate the $\mu$-minors $\delta_1, \ldots, \delta_{|n-m|+1}$. To be precise, when $m \leq n$, we will write
\[
\delta_i \coloneqq \begin{vmatrix}
x_{1,i} & x_{1,{i+1}} & \cdots \\
x_{2,i} & x_{2,{i+1}} & \cdots \\
\vdots & \vdots & \ddots
\end{vmatrix}, \quad \forall i=1,\ldots,|n-m|+1,
\]
whereas we will write
\[
\delta_{|n-m|+2-i} \coloneqq \begin{vmatrix}
x_{i,1} & x_{i,2} & \cdots \\
x_{i+1,1} & x_{i+1,2} & \cdots \\
\vdots & \vdots & \ddots
\end{vmatrix}, \quad \forall i=1,\ldots,|n-m|+1,
\]
when $m \geq n$. Next, we write their product as
\[
\Delta \coloneqq \alpha_1 \cdots \alpha_{\mu-1} \delta_1 \cdots \delta_{|n-m|+1} \beta_{\mu-1} \cdots \beta_1.
\]
\begin{theorem}[{\cite[\S 2]{SecciaKnutsonIdealsOfGenericMatrices}}] \label{thm.MainTheoremDeterminants}
With notation as above, suppose that $p>0$ and let $\Phi \: F_* S \to S$ be a Frobenius trace on $S$. Then, $\phi \coloneqq \Phi \cdot \Delta^{p-1} \in \sC_{1,S}$ is such that $\phi(F_*1) = 1$ and $I_t(\bm{x}), I_{t-1}(\bm{x}') \in I(S,\phi)$.
\end{theorem}
As mentioned above, Seccia's results are much stronger. She proves that ideals of $t$-minors of any submatrix of adjacent rows or adjacent columns are Knutson ideals of $\Delta$ and so they belong to $I(S,\phi)$. The results claimed in \autoref{thm.MainTheoremDeterminants} are much simpler and we give a proof below that partly deviates from Seccia's. But first, we establish the corollary we need for the purpose of this work.

\begin{corollary}
    \label{thm.FpurityDeterminantalPairs}
With notation as above, $(R_t^{m\times n},P_t^{m \times n})$ is $F$-pure if $p>0$.
\end{corollary}
\begin{proof}
    By \autoref{thm.MainTheoremDeterminants}, $I_t(\bm{x})$ and $I_{t-1}(\bm{x}')+I_t(\bm{x})$ belong to $I(S,\phi)$. Let $\varphi \: F_* R_t^{m \times n} \to R_t^{m \times n}$ be the induced map by $\phi$.  Then, $\varphi$ is a $P_t^{m\times n}$-map and $\varphi(F_* 1) =1$.
\end{proof}

To show \autoref{thm.MainTheoremDeterminants}, we will need the following observation.
\begin{lemma} \label{rem.ReductionToCOmpleteIntersection}
With notation as above, suppose that $m \leq n$ and write $\bm{x} = \begin{pmatrix}
    \bm{x}_1 & \cdots & \bm{x}_n \end{pmatrix}$, \ie $\bm{x}_i$ is the $i$-th column of $\bm{x}$. For all $i \in \{1, \ldots, n-m+1\}$, let $\gamma_i$ be the $m$-minor of $\bm{x}$ defined as follows:
   \[
\gamma_{i} \coloneqq \begin{vmatrix}
    \bm{x}_i & \bm{x}_{n-m+2} & \cdots & \bm{x}_n \end{vmatrix}.
   \] 
Then, we have the following prime decomposition in $S$:
\[
\sqrt{(\gamma_1, \ldots, \gamma_{n-m+1})}=I_{m}(\bm{x}) \cap I_{m-1}\begin{pmatrix}
    \bm{x}_{n-m+2} & \cdots & \bm{x}_n \end{pmatrix}.
\]
In particular, $I_{m}(\bm{x})$ is a minimal prime of $(\gamma_1, \ldots, \gamma_{n-m+1})$.
\end{lemma}
\begin{proof}
Set $h \coloneqq n-m+1 = \height I_m(\bm{x})$. By definition, $\gamma_i \in I_m(\bm{x})$ for all $i$. In computing $\gamma_i$, we may do it as a Laplace expansion along the first column of $\begin{pmatrix}
    \bm{x}_i & \bm{x}_{h+1} & \cdots & \bm{x}_n \end{pmatrix}$. This implies that each $\gamma_i$ belongs to the ideal generated by the $m-1$ minors of $\begin{pmatrix}
    \bm{x}_{h+1} & \cdots & \bm{x}_n \end{pmatrix}$, which is none other than $I_{m-1}\begin{pmatrix}
    \bm{x}_{h+1} & \cdots & \bm{x}_n \end{pmatrix}$. Hence, the inclusion ``$\subset$'' holds.

    To show the converse inclusion ``$\supset$,'' we use Hilbert's Nullstellensatz. Let $K$ be an algebraic closure of $\kay$. Let $A = \begin{pmatrix} v_1 & \cdots & v_n \end{pmatrix}$ be an $m \times n$ matrix over $K$ with vector columns $v_i \in K^m$.  Suppose that $\gamma_i(A)=0$ for all $i\in \{1,\ldots,h\}$. That is, for all $i\in \{1,\ldots,h\}$, the vectors $v_i,v_{h+1},\ldots,v_n \in K^m$ are linearly dependent. Suppose also that $A$ is not a $K$-point of the $K$-variety cut out by $I_{m-1}\begin{pmatrix}
    \bm{x}_{h+1} & \cdots & \bm{x}_n \end{pmatrix}$. That is, the vectors $v_{h+1},\ldots,v_n \in K^m$ are linearly independent. Therefore, every $v_i$ belongs to the $K$-span of $v_{h+1},\ldots,v_n$. Hence, the $K$-span of $v_1,\ldots, v_n$ has dimension $<m$ and so the rank of $A$ is $<m$. This means that $A$ is a $K$-point of the variety cut out by $I_m(\bm{x})$; as required.   
\end{proof}

\begin{proof}[Proof of \autoref{thm.MainTheoremDeterminants}]
The $F$-splitting part follows as in \cite{SecciaKnutsonIdealsOfGenericMatrices} by using \cite[Theorem 2]{KnutsonFrobeniusSplittingPointCountingDegeneration}. We explain next why it suffices to show that $I_{t-1}(\bm{x}') \in I(S,\phi)$, for all $t \geq 2$. To this end, let us consider the following prime decompositions:
\[
I_t(\bm{y}) + I_t(\bm{z})  = I_t(\bm{x}) \cap I_{t-1}(\bm{w})
\]
where $\bm{y}$ (resp. $\bm{z}$) is obtained by considering the first (resp. last) $m-1$ rows of $\bm{x}$ and $\bm{w}$ is given by excluding first and last row of $\bm{x}$. The inclusion ``$\subset$'' is clear, and the converse follows by comparing $K$-points in some algebraic closure $K$ of $\kay$; just as in the proof of \autoref{rem.ReductionToCOmpleteIntersection}.\footnote{These prime decompositions are also key in Seccia's work \cite{SecciaKnutsonIdealsOfGenericMatrices}.}

In particular, $I_t(\bm{x}) \in I(S,\phi)$ if $I_t(\bm{y}), I_t(\bm{z}) \in I(S,\phi)$. By iterating this process, we see that it suffices to prove that $I_t(\bm{x}'') \in I(S,\phi)$ for any submatrix $\bm{x}''$ of $\bm{x}$ consisting of $t$ adjacent rows. Next, we explain why we may further reduce to the case in which we take either the first $t$ rows or last $t$ rows.

Suppose $\bm{x}''$ is a submatrix of $t$ adjacent rows containing neither the first nor last row. Let $\bm{r}_1$ be the row preceding $\bm{x}''$ and $\bm{r}_2$ be the one proceeding it. Then, using the same type of prime decomposition we started with this proof, we have
\[
I_{t+1}\begin{pmatrix}
    \bm{r}_1 \\
    \bm{x}''
\end{pmatrix}
+ 
I_{t+1}\begin{pmatrix}
    \bm{x}'' \\
    \bm{r}_2
\end{pmatrix} = I_{t+1}\begin{pmatrix}
    \bm{r}_1 \\
    \bm{x}'' \\
    \bm{r}_2
\end{pmatrix} \cap I_t(\bm{x}'').
\]
Therefore, $I_t(\bm{x}'')$ is a $\phi$-ideal if so are both summands on the left hand side of the displayed equality. By iterating this process, we have reduced to the case in which $\bm{x}''$ consists either of the first $t$ rows or last $t$ rows; as required. By the symmetry, we may reduce further to the former case. This explains why it suffices to show $I_{t-1}(\bm{x}') \in I(S,\phi)$, which we do next.

Set $\p \coloneqq I_{t-1}(\bm{x}')$ and $h \coloneqq \height \p = n-t+2$. For all $i=1,\ldots,h$, define $\gamma_i \in \p$ as in \autoref{rem.ReductionToCOmpleteIntersection} with $\bm{x}'$ in place of $\bm{x}$. The inclusion $\phi(F_*\p) \subset \p$ can be verified after localizing at $\p$. That is, it suffices to prove that $\p S_{\p} \in I(S_{\p},\phi_\p)$ where $\phi_{\p} \: F_* S_{\p} \to S_{\p}$. By \autoref{rem.ReductionToCOmpleteIntersection}, the elements $\gamma_1, \ldots, \gamma_{h} \in \p S_{\p}$ define a system of parameters and so it suffices to show that
\[
\Delta^{p-1} \gamma_i \in \big(\gamma_1^p,\ldots,\gamma_{h}^p\big)S_{\p}, \quad \forall i\in  \{1,\ldots, h\}.
\]

Note that $\beta_{t-1}=\gamma_{h}$, so the case $i=h$ is immediate. For the remaining cases, it will be convenient to use the following notation. First, let us set $\beta_0 \coloneqq 1$. For every $i=\mu, \ldots, \nu \coloneqq \max \{m,n\}$, let $\bm{x}_i$ the $\mu \times \mu$ submatrix of $\bm{x}$ used to define the $\mu$-minor $\delta_{\nu +1-i}$. If $\bm{x}_{i}$ does not intersect every column of $\bm{x}$ (which only happens if $m < n$), let us expand it to the following $i\times i$ matrix:
\[
\tilde{\bm{x}}_{i} \coloneqq \begin{pmatrix}
    \bm{x}_i & \bm{x}'' &  \\
    \bm{0} & I_{i-\mu}
\end{pmatrix}
\]
where $\bm{x}''$ is the submatrix of $\bm{x}$ consisting of its last $i-\mu$ columns and  $I_{i-\mu}$ is the identity matrix of size $i-\mu$. We then set:
\[
\beta_i \coloneqq |\tilde{\bm{x}}_{i}| = |\bm{x}_{i}|=\delta_{\nu+1-i}.
\]

With the above in place, by using the so-called \emph{Sylvester's determinant identity} (see \cite{AkritasAkritasMalaschonkSylvesterIdentity} and the references therein), we see that
\[
\beta_{t-2}^k \beta_{t+k-1} = \begin{vmatrix}
\gamma_{h-k}&\cdots & \gamma_{h-2} & \gamma_{h-1} & \gamma_{h} \\
*&\cdots & * & * & * \\
\vdots &\ddots & \vdots & \vdots & \vdots \\
*&\cdots & * & * & *
\end{vmatrix}\eqqcolon \Gamma_{k}, \quad \forall k\in \{0,1,\ldots,h-1\},
\]
where the displayed matrix has size $(k+1) \times (k+1)$ and consists of the $(t-1)$-minors of the matrix we used to define $\beta_{t+k-1}$. Note that the matrix defining $\beta_{t-1}=\gamma_{h}$ sits in the right upper corner.

In particular, since $\beta_{t-2} \not\in \p$, we obtain the following equalities in $S_{\beta_{t-2}} \subset S_{\p}$:
\[
\Delta = \alpha_1 \cdots \alpha_{\mu-1} \beta_{\nu} \cdots \beta_1 =  \frac{\alpha_1 \cdots \alpha_{\mu-1} \beta_1 \cdots \beta_{t-2}  \Gamma_{h-1} \cdots \Gamma_0 }{\beta_{t-2}^{1+2+\cdots + (h-1)}}.
\]
Since there is an expansion of the form
\[
\Gamma_k^{p-1} = \sum_{i_{h-k}+ \cdots + i_h =p-1} c_{i_{h-k},\ldots,i_{h}} \gamma_{h-k}^{i_{h-k}} \cdots \gamma_{h}^{i_h},
\]
we obtain that
\[
(\Gamma_{h-1} \cdots \Gamma_0)^{p-1} \equiv c \gamma_1^{p-1} \cdots \gamma_h^{p-1}\bmod \big(\gamma_1^p,\ldots, \gamma_h^p\big)
\]
for some $c \in S$. Therefore,
\[
\Delta^{p-1} \equiv c' \gamma_1^{p-1} \cdots \gamma_h^{p-1} \bmod \big(\gamma_1^p,\ldots, \gamma_h^p\big)
\]
for some $c' \in S_{\beta_{t-2}} \subset S_{\p}$. Thus, $\Delta^{p-1} \gamma_i \in \big(\gamma_1^p,\ldots,\gamma_h^p\big) \subset S_{\p}$; as required.
\end{proof}
In the spirit of \cite{SecciaBinomialEdgeIdealsOfWeaklyClosedGraphs}, we may ask:
\begin{question}
   With notation as above, what are all the centers of $F$-purity of $(S,\phi)$?
\end{question}

\section{Proof of the Main Theorem}

In this section, we show our main result, which we recall next.

\begin{theorem} \label{thm.MainTheorem} 
Let $\kay$ be a field of characteristic $p\geq 0$ and $m,n \geq t \geq 2$ be a triple of integers. Let $(R_t^{m\times n},P_t^{m \times n})$ be a determinantal pure pair over $\kay$, with notation as in \autoref{subsec.DeterminantalPairs}. The following statements hold:
\begin{itemize}
    \item If $p>0$ then $(R_t^{m\times n},P_t^{m \times n})$ is purely $F$-regular.
    \item If $p=0$ and $m=n-1$ then $(R_t^{m\times n},P_t^{m \times n})$ is PLT.
    \item If $p=0$ then $(R_t^{m\times n},P_t^{m \times n})$ is of PLT-type.
\end{itemize}
\end{theorem}
Before getting to the proof of \autoref{thm.MainTheorem}, we establish a few reductions. First, we explain why we may assume that $(R_t^{m\times n},P_t^{m \times n})$ is log Gorenstein, \ie $m=n-1$.

\begin{proposition} \label{pro.reductiontoLoggorCase}
For a fixed $t \geq 2$, suppose that \autoref{thm.MainTheorem} holds for all $m,n \geq t$ such that $m=n-1$. Then, \autoref{thm.MainTheorem} holds for all $m,n \geq t$.  
\end{proposition}
\begin{proof}
We consider two cases: $m < n-1$ and $m \geq n$.

\textbf{Suppose that $m<n-1$:} We may enlarge $\bm{x}$ to an $(n-1)\times n$ matrix of indeterminates $\Tilde{\bm{x}}$ by adding $n-1-m$ generic rows. 
 Let $R_t^{(n-1) \times n} = \kay[\Tilde{\bm{x}}]/I_t(\tilde{\bm{x}})$. Since $I_t(\bm{x})\kay[\tilde{\bm{x}}] \subset I_t(\tilde{\bm{x}})$, we obtain a commutative diagram of $\kay$-algebras
\[
\xymatrix{
R_t^{m \times n} \ar@{^{(}->}[r] \ar[rd]_-{\id} & R_t^{(n-1) \times n} \ar@{->>}[d] \\
& R_t^{m \times n}
}
\]
where the vertical homomorphism is obtained by taking the quotient by the variables added to obtain $\Tilde{\bm{x}}$ from $\bm{x}$. Moreover, the prime ideal $\p_t^{m \times n} \subset R_t^{m \times n}$ extends to the prime ideal $\p_t^{(n-1) \times n} \subset R_t^{(n-1) \times n}$. Therefore, we may just apply \autoref{prop.SplitHomomorphismDescend} and  \cite[Theorem 2.10, Theorem A.1]{ZhuangDirecSummandsKLTSingularities} to conclude.

\textbf{Suppose that $m \geq n$:} We may enlarge $\bm{x}$ to an $m\times (m+1)$ matrix of indeterminates $\Tilde{\bm{x}}$ by adding $m-n+1$ generic columns. As before, letting $R_t^{m \times (m+1)} = \kay[\Tilde{\bm{x}}]/I_t(\tilde{\bm{x}})$ yields a commutative diagram of $\kay$-algebras
\[
\xymatrix{
R_t^{m \times n} \ar@{^{(}->}[r] \ar[rd]_-{\id} & R_t^{ m \times (m+1)} \ar@{->>}[d]^-{\pi} \\
& R_t^{m \times n}
}
\]
where the vertical homomorphism is obtained by annihilating the columns added to obtain $\Tilde{\bm{x}}$ from $\bm{x}$. However, $\p_t^{m\times n} \subset R_t^{m \times n}$ does not extend to $\p_t^{ m \times (m+1)} \subset R_t^{m \times (m+1)}$. Of course, we still have $\p_t^{m\times n} R_t^{m \times (m+1)} \subset \p_t^{ m \times (m+1)}$ and that $\pi$ sends $\p_t^{ m \times (m+1)}$ inside $\p_t^{m\times n}$. Hence, if $p>0$, we may use \autoref{prop.SplitHomomorphismDescend} to conclude that $(R_t^{m \times n}, P_t^{m \times n})$ is purely $F$-regular if so is $\big(R_t^{m \times (m+1)}, P_t^{m \times (m+1)}\big)$. 

If $p=0$, we need to be more careful in order to apply \cite{ZhuangDirecSummandsKLTSingularities}. Namely, we need to show that $\p_t^{ m \times (m+1)}$ is the only primary component of $\p_t^{m\times n} R_t^{m \times (m+1)}$ of height $1$. By height considerations, it is clear that $\p_t^{ m \times (m+1)}$ is a minimal prime of $\p_t^{m\times n} R_t^{m \times (m+1)}$ and in fact it contracts to $\p_t^{m\times n}$ (use $\pi$). Thus, we must show two things: the $\p_t^{ m \times (m+1)}$-primary component of $\p_t^{m\times n} R_t^{m \times (m+1)}$ is $\p_t^{ m \times (m+1)}$ and the other minimal primes of $\p_t^{m\times n} R_t^{m \times (m+1)}$ have height $\geq 2$. For notation ease, in what follows, we let $k > n$ be arbitrary and consider the pure extension $R_t^{m \times n} \subset R_t^{m \times k}$ with retraction $\pi \: R_t^{m \times k} \to R_t^{m \times n}$ given by annihilating the last $k-n$ columns (the case $k=m+1$ is the case discussed above).  Further, set $R \coloneqq R_t^{m \times n}$, $S \coloneqq R_t^{m \times k}$, $\p \coloneqq \p_t^{m \times n}$, and $\q \coloneqq \p_t^{m \times k}$

\begin{claim}
    The $\q$-primary component of $\p S$ is $\q$
\end{claim}
\begin{proof}[Proof of claim]
    The $\q$-primary component of $\p S$ is $(\p S_{\q})\cap S$, which is none other than $\q^{(e)}$ where $e$ is the ramification index of the extension of DVRs $R_{\p} \subset S_{\q}$. Thus, it suffices to show that an uniformizer of $R_{\p}$ is an uniformizer of $S_{\q}$. To this end, let us recall how to find an uniformizer of $R_{\p}$. First, let $\alpha=|\hat{\bm{x}}|$ be the $(t-2)$-minor of $\bm{x}$ where $\hat{\bm{x}}$ is the matrix given by the first $t-2$ rows and columns of $\bm{x}$ (if $t=2$ then $\alpha=1$). Let us invert $\alpha \notin \p$. By performing row operations, we then obtain
\[
\begin{pmatrix}
\hat{\bm{x}} & \bm{v}_{t-1} & \bm{v}_t &\cdots &\bm{v}_{n} \\
\bm{u}_{t-1} &x_{t-1, t-1} &x_{t-1, t} &\cdots &x_{t-1,n} \\
\bm{u}_{t} &x_{t, t-1} &x_{t, t} &\cdots &x_{t,n} \\
\vdots &\vdots &\vdots &\ddots & \vdots \\
\bm{u}_{m} &x_{m,t-1} &x_{m,t} &\cdots &x_{m,n}
\end{pmatrix}
\sim
\begin{pmatrix}
\hat{\bm{x}} & \bm{v}_{t-1} & \bm{v}_{t} &\cdots &\bm{v}_{n} \\
\bm{0} &y_{t-1, t-1} & y_{t-1, t} &\cdots &y_{t-1,n} \\
\bm{0} &y_{t, t-1} &y_{t, t}  &\cdots &y_{t,n} \\
\vdots &\vdots &\vdots &\ddots & \vdots \\
\bm{0} &y_{m,t-1} & y_{m,t} &\cdots &y_{m,n}
\end{pmatrix},
\]
where
\[
y_{i,j} = x_{i,j}-\bm{u}_i\hat{\bm{x}}^{-1} \bm{v}_j, \quad \hat{\bm{x}}^{-1}\coloneqq  \frac{1}{\alpha} \Adj \hat{\bm{x}},  
\]
and $\Adj \hat{\bm{x}}$ is the adjoint matrix of $\hat{\bm{x}}$---the transpose of the matrix of cofactors. The point is that both matrices share the same ideals of minors in $R_{\alpha}$ whenever the first $t-2$ rows are involved. For instance,
\[
\begin{vmatrix}
    \hat{\bm{x}} & \bm{v}_j  \\
\bm{u}_{t-1} &x_{{t-1}, j} 
\end{vmatrix}
= \begin{vmatrix}
    \hat{\bm{x}} & \bm{v}_j  \\
\bm{0} &y_{{t-1}, j} 
\end{vmatrix}
= \alpha y_{{t-1}, j}, \quad \forall j\in\{t-1,\ldots,n\}.
\]
Therefore, $\p R_{\alpha} = (y_{t-1,t-1}, \ldots, y_{t-1,n})$. Likewise,
\[
\alpha \begin{vmatrix}
y_{t-1, i} & y_{t-1, j}  \\
y_{t, i} &y_{t, j} 
\end{vmatrix} =
\begin{vmatrix} \hat{\bm{x}} & \bm{v}_i & \bm{v}_{j}  \\
\bm{0} &y_{{t-1}, i} & y_{t-1,j} \\
\bm{0} &y_{{t}, i} & y_{t,j}
\end{vmatrix} = 0, \quad \forall i,j\in\{t-1,\ldots,n\},
\]
as it is a $t$-minor. Hence,
\[
y_{t-1,i}y_{t,j} - y_{t-1,j}y_{t,i} = 0 \in R_{\alpha}.
\]
Thus, after inverting all the $y_{t,t-1},\ldots,y_{t,n} \notin \p R_{\alpha}$, we get that the $y_{t-1,t-1},\ldots,y_{t-1,n}$ are all multiples of one another. In other words,
\[
\p R_{\alpha y_{t,t-1}\cdots y_{t,n}} = (y_{t-1,j}), \quad \forall j \in \{t-1,\ldots,n\}.
\]
In particular, $y_{t-1,t-1}$ is a uniformizer of $R_{\p}$. Note that this is independent of $n$ and so we can choose a common uniformizer of $R_{\p}$ and $S_{\q}$; as required.
\end{proof}
\begin{claim}
The other minimal primes of $\p S$ have height $\geq 2$. Indeed,
\[
\sqrt{\p S} = \q \cap \big(I_{t-1}(\bm{x})S \big)
\]
and further $I_{t-1}(\bm{x})S \supsetneq I_{t-1}(\bm{x}'')S$ where $\bm{x}''$ is the matrix obtained by taking the first $t-1$ columns of $\bm{x}$ and so of $\tilde{\bm{x}}$. Of course, $I_{t-1}(\bm{x}'')S$ is a prime ideal of height $1$.\footnote{It generates $\Cl R$ and is linearly equivalent to $-\q$.}
\end{claim}
\begin{proof}[Proof of claim]
    It suffices to show the displayed equality of ideals. The inclusion ``$\subset$'' is clear. For the converse, we compare algebraic sets in some algebraic closure $K$ of $\kay$ and use Hilbert's Nullstellensatz. Let $A$ be an $m \times k$ matrix over $K$ of rank $<t$, \ie $A \in \bM_t^{m \times k}(K)$.  Let $u_1, \ldots, u_m \in K^{k}$ be the row vectors of $A$ and, for all $i=1,\ldots,m$, let $u'_i\in K^{n}$ be the row vector obtained by considering the first $n$-entries of $u_i$. Suppose that $A$ is a $K$-point of $V(\sqrt{\p S})$. That is, $\dim_K \langle u'_1, \ldots, u'_{t-1} \rangle_K < t-1 $. Suppose that $A$ is not a $K$-point of the variety cut out by $\q$. That is, the vectors $u_1, \ldots, u_{t-1} \in K^{k}$ are linearly independent and so the rank of $A$ is $t-1$. This means that $u_i\in \langle u_1, \ldots, u_{t-1} \rangle_K$ for all $i \in \{t,\ldots,m\}$. Then, the same holds with $u'$'s instead of $u$'s. Hence, $\dim_K \langle u'_1, \ldots, u'_{m} \rangle_K < t-1 $.  In other words, the rank of the matrix obtained by taking the first $n$ columns of $A$ has rank $<t-1$; as required.
\end{proof}
This proves \autoref{pro.reductiontoLoggorCase}.
\end{proof}
\begin{proposition} \label{pro.ReductiontoPosChar}
    If \autoref{thm.MainTheorem} holds for all $p>0$ then it holds for $p=0$.
\end{proposition}
\begin{proof}
    By \autoref{pro.reductiontoLoggorCase}, in proving \autoref{thm.MainTheorem} for $p=0$, we may reduce to the log Gorenstein case. The log Gorenstein case can be obtained by spreading out to positive characteristics by \cite[Theorem 4.6]{TakagiWatanabeFsingularities}. However, a determinantal pair of characteristic zero spreads out to determinantal pairs of positive characteristics. 
\end{proof}

We are going to need the following auxiliary pure pairs.
\begin{setup} \label{setup.AuxiliaryPairs}
Let us fix positive integers $s,t,m \geq 1$ such that $m \geq s,t$ and $s \geq t-1$. Set $n \coloneqq m+1$ and $k\coloneqq s-(t-1)\geq 0$.  Then consider an $m \times (n+k)$ matrix of variables
\[
\bm{y} \coloneqq \begin{pmatrix}
    \bm{x}' & \bm{z} \\
    \bm{w} & \bm{0}
\end{pmatrix}, \mbox{ where }
\bm{x} \coloneqq \begin{pmatrix}
    \bm{x}'   \\
    \bm{w} 
\end{pmatrix} 
\]
has size $m \times n$, $\bm{x}'$ has size $s \times n$ and so $\bm{z}$ has size $s \times k$ and $\bm{w}$ has size $(m-s) \times n$. Let us define:
\[
T \coloneqq \kay[\bm{y}]=S[\bm{z}], \quad R \coloneqq T/I_t(\bm{x}) = R_t^{m \times n} [\bm{z}], \quad \p \coloneqq I_s \begin{pmatrix}
    \bm{x}' & \bm{z}
\end{pmatrix}.
\]
Observe that $R$ is a normal integral domain. Note that the case $k=0$ is the case of determinantal pure pairs. However, the remaining cases $k=1,\ldots, m$ can be obtained from cases $k=0$ by successive localizations along entries of $\bm{w}$. We make this precise in the proof of the following proposition establishing that $P \colonequals V(\p) \subset \Spec R$ is a prime divisor.
\end{setup}

\begin{proposition} \label{prop.Height1}
    Work in \autoref{setup.AuxiliaryPairs}. The ideal $\p \subset R$ is a prime ideal of height $1$.
\end{proposition}
\begin{proof}
Let $w$ be an entry of $\bm{w}$, say $w\coloneqq w_{m-s,n}=x_{m,n}$. After inverting $w$, we have:
\[
\bm{y} = \begin{pmatrix}
    \bm{x}'' & \bm{v} & \bm{z} \\
    \bm{w}' & \bm{c} & \bm{0} \\
    \bm{r} & w  & \bm{0}
\end{pmatrix}
\sim 
\begin{pmatrix}
    \tilde{\bm{x}}' \coloneqq \bm{x}'' - w^{-1} \bm{v}\bm{r} & \bm{0} & \bm{z} \\
    \tilde{\bm{w}} \coloneqq \bm{w}'- w^{-1} \bm{c}\bm{r} & \bm{0} & \bm{0} \\
    \bm{0} & w  & \bm{0}
\end{pmatrix},
\quad \tilde{\bm{x}} \coloneqq \begin{pmatrix}
    \tilde{\bm{x}}'   \\
    \tilde{\bm{w}}
\end{pmatrix}
\]
and
\[
I_t(\bm{x})T_w= I_{t-1}(\tilde{\bm{x}})T_w, \quad I_s \begin{pmatrix}
    \bm{x}' & \bm{z}
\end{pmatrix} T_w= I_s \begin{pmatrix}
    \tilde{\bm{x}}' & {\bm{v}} & \bm{z}
\end{pmatrix} T_w.
\]
Hence, letting 
\[
\tilde{\bm{z}} \coloneqq \begin{pmatrix}
     \bm{v} & \bm{z}
\end{pmatrix},
\quad \tilde{\bm{y}} \coloneqq \begin{pmatrix}
    \tilde{\bm{x}}' & \tilde{\bm{z}} \\
    \tilde{\bm{w}} & \bm{0}
\end{pmatrix}
\] yields:
\[
R_w = \left(\frac{\kay[\tilde{\bm{y}}]}{I_{t-1}(\tilde{\bm{x}}) } \right) \big[w, w^{-1}, \bm{r}, \bm{c}\big] = \tilde{R}\big[w, w^{-1}, \bm{r}, \bm{c}\big]
\]
and 
\[
\p R_w = \big(I_s \begin{pmatrix}
    \tilde{\bm{x}}' & \tilde{\bm{z}}
\end{pmatrix}  \tilde{R} \big)\big[w, w^{-1}, \bm{r}, \bm{c}\big] = \tilde{\p}\big[w, w^{-1}, \bm{r}, \bm{c}\big].
\]
Observe that $(\tilde{R}, \tilde{\p})$ is none other than the case $(m-1,t-1,s)$ and so we have increased the value of $k$ by $1$. Thus, by induction on $k$, we conclude that $\height \p = 1$. 
\end{proof}
In what follows, we are going to study the pure pairs $(R,P)$ where $P = V(\p) \subset \Spec R$. First, we show that they are log Gorenstein as $n=m+1$. This is obtained by induction on $k$ and localizing along entries of $\bm{z}$.
\begin{proposition} \label{pro.LogGorensteinVerification}
Work in \autoref{setup.AuxiliaryPairs}. Then, $-P$ is a canonical divisor on $R$.    
\end{proposition}
\begin{proof}
We proceed by induction on $k$. The case $k=0$ is the determinantal case. Assume $k\geq 1$. Observe that $\p \subset (\bm{z})R$, \ie $\p$ is contained in the prime ideal generated by the entries of $\bm{z}$. If $s=1$, then $t=1$ and $k=1$. In that case, $(R,\p)=\big(\kay[z],(z)\big)$, where the proposition is trivial. Thus, we may assume that $s \geq 2$. In that case, $\p \subsetneq (\bm{z})R$ and so $\height (\bm{z})R \geq 2$. Since the statement can be proved in codimension $1$, we may prove it away from the closed subset cut out by $(\bm{z})R$ and so after inverting each entry of $\bm{z}$. Without loss of generality, we may localize along $z\coloneqq z_{1,k}$. After inverting $z$, we obtain:
\[
\bm{y} = \begin{pmatrix}
    \bm{u} & \bm{r} & z \\
    \bm{x}'' & \bm{z}' & \bm{c} \\
    \bm{w} & \bm{0}  & \bm{0}
\end{pmatrix}
\sim 
\begin{pmatrix}
    \bm{0} & \bm{0} & z \\
   \tilde{\bm{x}}' \coloneqq \bm{x}'' - z^{-1} \bm{c}\bm{u}  &   \tilde{\bm{z}} \coloneqq \bm{z}' - z^{-1} \bm{c}\bm{r} & \bm{0} \\
    \bm{w} & \bm{0}  & \bm{0}
\end{pmatrix},
\]
and so we set
\[
\tilde{\bm{x}} \coloneqq \begin{pmatrix}
    \bm{u} \\
    \tilde{\bm{x}}'   \\
    \bm{w} 
\end{pmatrix},
\quad \tilde{\bm{y}} \coloneqq \begin{pmatrix}
    \bm{u} & \bm{0} \\
    \tilde{\bm{x}}' & \tilde{\bm{z}}  \\
    \bm{w} & \bm{0} 
\end{pmatrix}.
\]
Then,
\[
I_t(\bm{x})T_z = I_{t}(\tilde{\bm{x}})T_z, \quad I_s \begin{pmatrix}
    \bm{x}' & \bm{z}
\end{pmatrix} T_z = I_{s-1} \begin{pmatrix}
    \tilde{\bm{x}}' & \tilde{\bm{z}}
\end{pmatrix} T_z. 
\]
In this way,
\[
R_z = \big( \kay[\tilde{\bm{y}}]/I_t(\tilde{\bm{x}}) \big)\big[ z, z^{-1}, \bm{r}, \bm{c}\big] = \tilde{R}\big[ z, z^{-1}, \bm{r}, \bm{c}\big]
\]
and
\[
\p R_z = \big( I_{s-1} \begin{pmatrix}
    \tilde{\bm{x}}' & \tilde{\bm{z}}
\end{pmatrix} \tilde{R}\big) \big[ z, z^{-1}, \bm{r}, \bm{c}\big]
= \tilde{\p}\big[ z, z^{-1}, \bm{r}, \bm{c}\big] 
\]
where $(\tilde{R}, \tilde{\p})$ is the case $(m-1,t,s-1)$ and so we have decreased $k$ by $1$. The result then follows by induction.
\end{proof}
In this way, we have verified that $(R,P)$ satisfies the setup of \autoref{cor.CheckingForFpurity}. We show next that it also satisfies conditions (b) and (c). In other words, the pure $F$-regularity of $(R,P)$ is equivalent to condition (a) in \autoref{cor.CheckingForFpurity} and so to any of the other conditions in \autoref{rem.SettingVariablesEqualTo1}. Say, we may check for pure $F$-regularity after setting $y_{i,j}=1$ for every entry of $\bm{y}$ such that $y_{i,j} \notin \p$.
\begin{proposition} \label{pro.SimplificationInductivehypothesis}
    Work in \autoref{setup.AuxiliaryPairs} and suppose that $\Char \kay = p>0$. Then, conditions (b) and (c) in \autoref{cor.CheckingForFpurity} hold for $(R, P)$. 
\end{proposition}
\begin{proof}
Observe that $R$ is $F$-regular as it is a polynomial algebra over an $F$-regular ring. To see why $R/\p$ is $F$-regular, note that
\[
R/\p = R \otimes_{\kay} T/ I_s \begin{pmatrix}
    \bm{x}' & \bm{z} 
\end{pmatrix},
\]
and $ T/ I_s \begin{pmatrix}
    \bm{x}' & \bm{z} 
\end{pmatrix}$ is $F$-regular by the same reason $R$ is. This checks (b).

We verify (c) next. The case $k=0$ was done in \autoref{thm.FpurityDeterminantalPairs}. In fact, if $k=0$, there is $\phi \: F_* T \to T$ such that $\phi(F_*1)=1$ and $I_t(\bm{x}), I_s \begin{pmatrix}
    \bm{x}' & \bm{z} 
\end{pmatrix} \in I(T,\phi)$. We prove the same holds for all $k$ by induction on $k$. Just as in the proof of \autoref{prop.Height1}, we may enlarge $\bm{y}$ to a matrix
\[
\tilde{\bm{y}} \coloneqq \begin{pmatrix}
    \bm{x}' & \bm{v} & \bm{z}' \\
    \bm{w} & \bm{c} & \bm{0} \\
    \bm{r} & w & \bm{0}
\end{pmatrix}
\]
where $\bm{v}$ is the first column of $\bm{z}$. That is, $\bm{y}$ is obtained from $\tilde{\bm{y}}$ by deleting the last row and setting $\bm{c}=0$. By induction, setting $\tilde{T} \coloneqq \kay[\tilde{\bm{y}}]$, there is $\tilde{\phi} \: F_* \tilde{T} \to \tilde{T}$ such that $\tilde{\phi}(F_*1)=1$ and
\[
I_{t+1}\begin{pmatrix}
    \bm{x}' & \bm{v} \\
    \bm{w} & \bm{c} \\
    \bm{r} & w
\end{pmatrix}, I_s \begin{pmatrix}
    \bm{x}' & \bm{z} 
\end{pmatrix} \in I\big(\tilde{T},\tilde{\phi}\big)
\]
as this decreases $k$ by $1$. Now, let $\phi \: F_* T \to T$ be the composition of $R$-linear maps
\[
\phi \: F_* T \to F_* \tilde{T} \xrightarrow{\tilde{\phi}} \tilde{T} \twoheadrightarrow \tilde{T}/(\bm{r},\bm{c},w-1) \cong T.
\]
As in \autoref{prop.SplitHomomorphismDescend}, one readily verifies that this map has the required properties as
\[
I_{t+1}\begin{pmatrix}
    \bm{x}' & \bm{v} \\
    \bm{w} & \bm{c} \\
    \bm{r} & w
\end{pmatrix} = I_{t}(\bm{x}) \bmod (\bm{r},\bm{c},w-1).
\]
\end{proof}

By \autoref{pro.ReductiontoPosChar} and \autoref{pro.reductiontoLoggorCase}, \autoref{thm.MainTheorem} follows from the following proposition.
\begin{proposition} \label{pro.mainProposition}
    Work in \autoref{setup.AuxiliaryPairs} and suppose that $\Char \kay = p>0$. Then, 
    $(R,P)$ is purely $F$-regular.
\end{proposition}
\begin{proof}
We proceed by induction on $s \geq 1$ and treat the base case next.
    
\begin{proof}[Proof of the base case $s=1$]
Since $s \geq t-1$, either $t=1$ or $t=2$. In the former case, $(R,\p) = (\kay[z],(z))$, which is clearly purely $F$-regular (see \autoref{rem.RegularAmbient}). Suppose now that $t=2$. By \autoref{pro.SimplificationInductivehypothesis}, it suffices to show that $(R,P)$ is purely $F$-regular after inverting each $y_{i,j} \notin \p$. Since $t=2$, $y_{i,j} \notin \p$ implies that $y_{i,j}$ is an entry of $\bm{w}$. In the proof of \autoref{prop.Height1}, we saw how, by inverting an entry of $\bm{w}$, we may apply \autoref{add_variable} and \autoref{test_ideal_localization} to reduce the case $(m,t=2,s=1)$ to the case $(m-1,t=1,s=1)$; and so we are done.  
\end{proof}

\begin{proof}[Proof of the inductive step] Let us assume that $s \geq 2$ and that $(R,P)$ is purely $F$-regular for all smaller values of $s$. To prove that $(R,P)$ is purely $F$-regular in this case, we proceed by induction on $t \geq 1$. The base case $t=1$ is the case $(R,\p)=(\kay[\bm{z}], (|\bm{z}|))$ where $\bm{z}$ has size $s \times s$. This case follows from \autoref{rem.RegularAmbient} as $\kay[\bm{z}]/(|\bm{z}|)$ is $F$-regular. Thus, we may further assume that $t \geq 2$ and the pure $F$-regularity of $(R,P)$ for all smaller values of $t$ as well.

Now, by \autoref{pro.SimplificationInductivehypothesis}, it suffices to show that $(R,P)$ is purely $F$-regular after inverting (equivalently, setting equal to $1$) each nonzero entry of $\bm{y}$. There are three cases depending on whether such entry belongs to either $\bm{x}'$, $\bm{w}$, or $\bm{z}$. 

Suppose that we invert an entry of $\bm{w}$ (if any). By the computation in the proof of \autoref{prop.Height1}, we may apply \autoref{add_variable} and \autoref{test_ideal_localization} to obtain the reduction $(m,t,s) \rightsquigarrow (m-1,t-1,s)$, and so we are done by the inductive hypothesis on $t$.

Suppose now that we invert an entry of $\bm{z}$ (if any). As in the proof of \autoref{pro.LogGorensteinVerification}, we may apply \autoref{add_variable} and \autoref{test_ideal_localization} to obtain the reduction $(m,t,s) \rightsquigarrow (m-1,t,s-1)$, and so we are done by the inductive hypothesis on $s$.

Finally, we invert an entry of $\bm{x}'$. Equivalently, set $x\coloneqq x_{1,1}=1$. Then,
\[
\bm{y}=\begin{pmatrix}
1 & \bm{r} & \bm{u} \\
\bm{c} & \bm{x}'' & \bm{z}'\\
\bm{v} & \bm{w}' & \bm{0}
\end{pmatrix}
\sim
\begin{pmatrix}
1 & \bm{0} & \bm{0} \\
\bm{0} & \tilde{\bm{x}}'\coloneqq \bm{x}'' -  \bm{r} \bm{c} & \tilde{\bm{z}} \coloneqq \bm{z}' -  \bm{u} \bm{v}\\
\bm{0} & \tilde{\bm{w}} \coloneqq \bm{w}' -  \bm{r} \bm{v} & \bm{0}
\end{pmatrix}.
\]
Then, by setting
\[
\tilde{\bm{x}} \coloneqq \begin{pmatrix}
    \tilde{\bm{x}}'   \\
    \tilde{\bm{w}}
\end{pmatrix}, \quad \tilde{\bm{y}} \coloneqq \begin{pmatrix}
 \tilde{\bm{x}}' & \tilde{\bm{z}} \\
  \tilde{\bm{w}} & \bm{0}
\end{pmatrix}
\]
we  obtain that $I_t(\bm{x}) = I_{t-1}(\tilde{\bm{x}})$ and $I_s \begin{pmatrix}
    \bm{x}' & \bm{z} 
\end{pmatrix}= I_{s-1} \begin{pmatrix}
    \tilde{\bm{x}}' & \tilde{\bm{z}} 
\end{pmatrix}$. Therefore,

\[
R/(x-1)  = \big(\kay[\tilde{\bm{y}}]/I_{t-1}(\tilde{\bm{x}})\big)\big[\bm{c}, \bm{r}, \bm{u}, \bm{v}\big]    
\]
and
\[
\p R/(x-1) = \big(I_{s-1} \begin{pmatrix}
    \tilde{\bm{x}}' & \tilde{\bm{z}} 
\end{pmatrix}\kay[\tilde{\bm{y}}]/I_{t-1}(\tilde{\bm{x}})\big)\big[\bm{c}, \bm{r}, \bm{u}, \bm{v}\big].    
\]
We are done by using the inductive hypothesis on $s$ together with \autoref{add_variable}. 
\end{proof}
This proves \autoref{pro.mainProposition}.
\end{proof}

\subsection*{Conflict of interest statement:} On behalf of all authors, the corresponding author states that there is no conflict of interest.
\bibliographystyle{skalpha}
\bibliography{MainBib}

\end{document}